\documentclass[a4paper]{article}
\usepackage{graphicx}
\usepackage{amsmath,amsthm,amssymb,enumerate,tikz-cd}
\usepackage{thm-restate}
\usepackage{mathtools}
\usepackage{hyperref}
\usepackage[capitalize]{cleveref}
\DeclareMathOperator{\Ker}{Ker}

\newtheorem{THM}{Theorem}[section]
\newtheorem{LEM}[THM]{Lemma}
\newtheorem{COR}[THM]{Corollary}

\newtheorem{PROP}[THM]{Proposition}

\def\specrelabove#1#2{\mathrel{\mathop{\kern0pt #1}\limits^{#2}}}
\def\specrelbelow#1#2{\mathrel{\mathop{\kern0pt #1}\limits_{#2}}}

\newcommand\N{\mathbb N}

\newcommand\R{\mathbb R}
\newcommand\Z{\mathbb Z}
\newcommand\join{\vee}
\newcommand\meet{\wedge}

\def\lowfwd #1#2#3{{\mathop{\kern0pt #1}\limits^{\kern#2pt\raise.#3ex \vbox to 0pt{\hbox{$\scriptscriptstyle\rightarrow$}\vss}}}}

\def\lowbkwd #1#2#3{{\mathop{\kern0pt #1}\limits^{\kern#2pt\raise.#3ex
\vbox to 0pt{\hbox{$\scriptscriptstyle\leftarrow$}\vss}}}}
\def\vp{\lowfwd p{1.5}2}
\def\pv{\lowbkwd p{1.5}2}

\def\vr{\lowfwd r{1.5}2}

\def\vS{\vec S}

\def\vSdash{{\mathop{\kern0pt S\lower-1pt\hbox{${}
     \scriptstyle'$}}\limits^{\kern2pt\raise.1ex
     \vbox to 0pt{\hbox{$\scriptscriptstyle\rightarrow$}\vss}}}}
\def\vs{\lowfwd s{1.5}2}
\def\sv{\lowbkwd s{1.5}2}
\def\vsdash{{\mathop{\kern0pt s\lower.5pt\hbox{${}
     \scriptstyle'$}}\limits^{\kern0pt\raise.02ex
     \vbox to 0pt{\hbox{$\scriptscriptstyle\rightarrow$}\vss}}}}
\def\svdash{{\mathop{\kern0pt s\lower.5pt\hbox{${}
     \scriptstyle'$}}\limits^{\kern0pt\raise.02ex
     \vbox to 0pt{\hbox{$\scriptscriptstyle\leftarrow$}\vss}}}}

\def\vt{\lowfwd t{1.5}2}

\def\vV{\vec V}

\def\vU{\vec U}
 \def\vv{\lowfwd v{1.5}2}
\def\vvback{\lowbkwd v{1.5}2}

\def\vx{\lowfwd x{1.5}2}
\def\xv{\lowbkwd x{1.5}2}
\def\vxdash{{\mathop{\kern0pt x\lower.5pt\hbox{${}
     \scriptstyle'$}}\limits^{\kern0pt\raise.02ex
     \vbox to 0pt{\hbox{$\scriptscriptstyle\rightarrow$}\vss}}}}
\def\xvdash{{\mathop{\kern0pt x\lower.5pt\hbox{${}
     \scriptstyle'$}}\limits^{\kern0pt\raise.02ex
     \vbox to 0pt{\hbox{$\scriptscriptstyle\leftarrow$}\vss}}}}
\def\vX{\vec X}
\def\vY{\vec Y}
\def\vy{\lowfwd y{1.5}2}
\def\yv{\lowbkwd y{1.5}2}

\def\vcB{\vec{\mathcal{B}}}

\def\es{\emptyset}
\def\sub{\subseteq}

\def\sm{\smallsetminus}

\def\partialone{{\partial}}
\def\deltanought{{\delta}}
\def\gammanought{{\gamma}}
\def\gammaone{{\gamma}}

\newcommand{\pull}{\prescript{\triangleright}{}}
\newcommand{\pullE}{\prescript{X}{}}
\newcommand{\pullright}{\prescript{\triangleleft}{}}
\newcommand{\abs}[1]{\left |#1\right |}

\newcommand\COMMENT[1]{}
\def\?#1{\vadjust{\vbox to 0pt{\vss\vskip-8pt\leftline{%
     \llap{\hbox{\vbox{\pretolerance=-1
     \doublehyphendemerits=0\finalhyphendemerits=0
     \hsize30truemm\tolerance=10000\small
     \lineskip=0pt\lineskiplimit=0pt
     \rightskip=0pt plus16truemm\baselineskip8pt\noindent
     \hskip0pt        
     #1\endgraf}\hskip7truemm}}}\vss}}}

\title{Duality and tangles of set separations}
 \author{Reinhard Diestel \and Christian Elbracht \and Joshua Erde\and Maximilian Teegen}
 \date{}

\begin{document}
\abovedisplayshortskip=-3pt plus3pt
\belowdisplayshortskip=6pt

\maketitle

\begin{abstract}\noindent
  Applications of tangles of connectivity systems suggest a duality between these, in which for two sets $X$ and~$Y\!$ the elements $x$ of~$X$ map to subsets $Y_x$ of~$Y\!$, and the elements $y$ of~$Y\!$ map to subsets $X_y$ of~$X$, so that $x\in X_y$ if and only if $y\in Y_x$ for all $x\in X$ and $y\in Y\!$. We explore this duality, and relate the tangles arising from the dual systems to each other.
   \end{abstract}

\section{Introduction}

Tangles are a relatively novel, indirect, notion of cluster in discrete structures. While originally introduced by Robertson and Seymour~\cite{GMX} for graphs and matroids, they have since been generalised to a much wider range of settings~\cite{TangleBook,ASS}.\looseness=-1

Rather than specifying a cluster by naming its elements, tangles describe clusters indirectly by a set of pointers. This makes them particularly well suited to capturing `fuzzy' clusters: clusters that appear obvious when viewed from a distance, but where allocating individual points to any of them can present a problem.

More precisely, if $C$ is a cluster in a data set~$V\!$, a~tangle will capture its location by orienting some suitable bipartitions of~$V\!$ towards their side that contains most of~$C$. If $S$ is a set of bipartitions of~$V\!$ each oriented towards~$C$ in this way, we say that the partitions in~$S$ are oriented \emph{consistently\/}: not arbitrarily, but all towards that fixed cluster~$C$.

Clearly, this can be done reliably only if each~$s\in S$ splits $C$ unevenly, if one of its sides contains noticeably more of~$C$ than the other. If $s$ has this property with respect to every cluster in~$V\!$, let us call it a {\em bottleneck partition\/} of~$V\!$.\looseness=-1

The idea behind tangles, now, is to offer axiomatic definitions of both `bottle\-neck partitions' and `consistent' that do not refer to any pre-conceived clusters but still bear out this idea: definitions such that most concrete examples of a set~$C$ of points in a dataset~$V\!$ which most people would intuitively see as a cluster are divided unevenly by those bipartitions of~$V\!$ which the axioms designate as `bottleneck partitions', and hence orient them, and such that the orientations of the bottleneck separations designated as `consistent' are those induced in this way by some intuitively perceived cluster. With such axiomatic definitions of bottleneck partitions, and of consistency for their collective orientations, in place, a~{\em tangle\/} is then defined as a consistent orientation of all the bottleneck partitions.

In graphs, there are some natural notions of bottleneck partitions. For example, one could take the vertex bipartitions that are crossed by few edges.%
  \footnote{Tangles orienting these are known as {\em edge-tangles\/} of the graph. Its {\em vertex-tangles\/}, those considered by Robertson and Seymour, orient bipartitions of the edge set, the bottleneck partitions being those for which few vertices are incident with edges on both sides.}
  For bipartitions of an arbitrary data set it is not so clear how to identify a set of bipartitions that will split any perceived cluster unevenly, and can hence serve as `bottleneck partitions', without referring to such clusters. Indeed one of the challenges in defining tangles in such a general setting lies in coming up with a notion of bottleneck partition that satisfies the requirements of tangle theory, such as submodularity.

In a typical clustering application this is facilitated by the existence of a natural {\em similarity function\/} $\sigma\colon V^2\to\N$, which assigns to a pair $(u,v)$ of points a large value if these points are deemed to be similar and should therefore be clustered together. For points in the Euclidean plane, for example, this could be the multiplicative inverse of their Euclidean distance. Note that although any such function appeals to some extrinsic notion of similarity of pairs of points, assuming the existence of such a function is much weaker than knowing how to cluster them, at least in the sense of partitioning~$V\!$ into `groups of similar points'. Indeed the rationale of tangles is that they offer a way to cluster~$V\!$ without the need to come up with such a partition.%
   \footnote{However, tangles give rise to a `fractional' such partition, which assigns every point $v$ to every tangle~$\tau$ with a weight between 0 and~1: the proportion of the bottleneck partitions for which $v$ lies on the side towards which $\tau$ orients them.}

As an example, consider the set $V\!$ of items for sale in an online shop. Two items $u,v\in V\!$ might be considered similar if they are often bought together. To measure this, we consult our sales records and let $P$ be the set of last year's purchases: every $p\in P$ partitions~$V\!$ into the set~$\vp\sub V\!$ of items bought in that purchase and the set~$\pv\sub V\!$ of items not bought. The similarity of two items $u$ and~$v$ might then be defined as the number $\sigma(u,v) := |\{\,p\in P \mid u,v\in\vp\,\}|$ of purchases that included both $u$ and~$v$. We can then define the {\em order\/} of a bipartition $s=\{A,B\}$ of~$V\!$ as the sum $\sum_{a\in A, b\in B}\sigma(a,b)$ over all pairs $a\in A$ and $b\in B$, and call $s$ a {\em bottleneck partition\/} if its order is below some threshold~$k$ (which we are free to set): if it does not divide many similar items. A~{\em $k$-tangle\/}, then, will be a {\em consistent\/} way of orienting all the bipartitions of~$V\!$ of order~$<k$, with respect to some formal notion of consistency depending on the type of tangle we consider.%
   \footnote{The formal definition of consistency is key to any notion of a tangle~\cite{TangleBook,ASS}. There are some minimum requirements, but within these there is room for choices depending on the intended application, giving rise to different types of tangle. Consistency is always defined in such a way that it is satisfied by any intuitively obvious cluster $C\sub V\!$ of points, in that orienting all the partitions of order~$<k$ towards where most of~$C$ lies will meet this definition.}
   The order $k$ of a $k$-tangle can be thought of as a measure of the \emph{coherence} of the cluster identified by this tangle: roughly, a tangle of order $k$ cannot be split into two roughly equal sized parts by a separation of order~$<k$.

In the example just discussed, we have outlined a way of clustering the items $v\in V\!$ of our online shop by appealing to the set $P$ of last year's purchases, each of which we interpreted as a bipartition of~$V\!$. (These were not the bottleneck partitions that our tangles ended up orienting, but bipartitions of~$V\!$ nonetheless: those we needed in order to determine which other bipartitions~$s$ we wished to consider as bottlenecks and thus have oriented by our tangles.) In the same way, however, we can view the shop items as partitions of the set~$P$ of purchases: every $v\in V\!$ partitions~$P$ into the set $\vv\sub P$ of purchases that included~$v$ and the set~$\vvback\sub P$ of purchases that did not include~$v$. These partitions of~$P$, then, can be used to define a similarity measure on the pairs of purchases, where two purchases $p,q\in P$ are deemed to be similar if they share many items. Formally, we would set $\sigma(p,q) := |\{\,v\in V \mid p,q\in \vv\,\}|$, in analogy to our earlier scenario. And once more, we can now proceed to use the similarity function on~$P\times P$ to define bottleneck partitions of~$P$, and consider the tangles of those partitions.

The two ways of describing the sales history of our online shop just outlined are not only analogous, they are dual in the formal sense that $v\in\vp$ if and only if $p\in\vv$. Their tangles are also related, but they are not the same. While the tangles we considered first capture fuzzy clusters of shop items that are similar in one way or another (with a tangle each for different aspects under which items can be similar), the tangles we considered afterwards describe fuzzy clusters of like-minded customers,%
   \footnote{Let us assume for simplicity that these are distinct for different purchases. This can easily be achieved by congregating the purchases of a given customer into one.}
   again with a tangle each for every aspect of taste or other type of our customers' purchasing motivation.
   
   An application of tangles viewed most naturally as those of a dual rather than a primal system will be given in~\cite{Big5}. Whether or not we view a separation system as primal or dual is technically just a matter of choice, but this choice has ramifications. For example, the order function used in~\cite{Big5}, which is based on the information-theoretic notion of {\em mutual information\/} based on discrete entropy~\cite{ShannonEntropy}, has a direct interpretation in which the technical meaning of `information' coincides with its natural meaning. There is no such natural entropy-based order function for primal separation systems; see~\cite{TangleOrder} for more.

Our aim is to explore this duality between the bipartitions $p\in P$ of~$V\!$ and the bipartitions $v\in V\!$ of~$P$, and to investigate how their tangles may be related. In this paper we take a first step towards this goal by relating, more generally, arbitrary partitions $(A,B)$ of $V\!$ and $(C,D)$ of~$P$, and see how their tangles are related. There is a natural way to do this: given~$(A,B)$, let $C$ be the set of purchases which contain more items from~$A$ than items from~$B$, and $D$ the set of purchases containing more items from~$B$ than from~$A$.%
   \footnote{Our formal setup will take care also of purchases for which equality holds here, see below.}
   Similarly, for any oriented bipartition $(C,D)$ of $P$ we have a corresponding oriented bipartition $(A,B)$, where $A$ is the set of items included in more purchases from~$C$ than from~$D$, and similarly for~$B$. Whilst these two operations are not inverse to each other, they are closely related. And we will show that they relate the tangle structure of the bipartitions of $V\!$ to that of the bipartitions of~$P$. 

We begin in Section~\ref{Prelim} by recalling the fundamental terminology and instances of separation systems. More details about separation systems, and in particular about their tangles, can be looked up in~\cite{ASS}. In Section~\ref{SetPartitions} we then state the duality of separation systems more formally.

Section~\ref{Tangles}, which contains our main results, relates the tangles of a pair of separation systems as earlier to each other. That is, we shall find natural order functions on the separations of $V\!$ and of~$P$ such that the correspondence between these separations outlined above gives rise to a correspondence between their tangles. More precisely, every $4k$-tangle of~$V\!$ will give rise to a $k$-tangle of~$P$, and vice versa.

Moreover this correspondence is in some sense idempotent in that, if we go from a tangle of~$V\!$ to a tangle of~$P$ and back again, we recover a restriction of the original tangle to a lower order. Interpreted in the setting of our online shop example, where~$V\!$ is the set of items in the shop and $P$ is the set of purchases in a given year, this implies that any coherent enough cluster of items will give rise to a cluster of purchases, and vice versa. Due to the idempotence of our correspondence between separations of~$V$ and~$P$, these two clusters should be related.

In Section~\ref{Hom}, finally, we note that the duality of set partitions can be cast in an algebraic framework, by defining a natural boundary operator on the oriented bipartitions of a given set. It then becomes an instance of algebraic duality, as between the homology and cohomology defined by this boundary operator. This gives rise to a host of questions that can be addressed purely algebraically. This is done in~\cite{HypergraphHomology}, while questions that are more directly related to clustering and tangles are addressed here. These include the introduction of an inner product of set partitions, whose associated notion of orthogonality agrees with geometric orthogonality in certain simple cases. As this inner product is defined for arbitrary finite set partitions, such as in our online shop example, it might therefore deserve further study.

\section{Basic terminology}\label{Prelim}
Throughout this paper, whenever we speak of `partitions' of a set we shall mean bipartitions, that is, partitions into two disjoint subsets. These subsets are allowed to be empty, although sometimes we may require that they are not.

We use the basic notions of abstract separation systems and their tangles as defined in~\cite{ASS}. For the convenience of the reader, we recall those of these concepts that are essential to this paper in this section.

An \emph{(abstract) separation system} is a partially ordered set $(\vS,\le)$ with an
order-reversing involution~$^\ast$. That is to say, $\vr\le\vs$ if and only if $\vr^\ast\ge \vs^\ast$ for all $\vr,\vs\in\vS$. The image of $\vs$ under $^\ast$ is usually denoted as $\sv$.
The elements of $\vS$ are called \emph{oriented separations}, and $\sv$ is the \emph{inverse} of~$\vs$. Note that a given element of~$\vS$ can be denoted as either $\vs$ or~$\sv$ according to context; there are no `default orientations'.

An oriented separation $\vs$ together with its inverse $\sv$ form an \emph{unoriented separation} $s = \{\vs, \sv\}$, and we say that $\vs$ and $\sv$ are \emph{orientations} of $s$. The set of all unoriented separations of a given separation system~$\vS$ is denoted as $S$.

If such a separation system $(\vS,\le)$ happens to be a lattice, that is, if there is a supremum $\vs\join \vt$ and an infimum $\vs\meet \vt$ for any two $\vs,\vt\in \vS$, then this separation system is a \emph{universe of separations}.

An {\em order function\/} on a universe~$\vU$ of separations is any real-valued function $\abs{\cdot}$ satisfying $0\le \abs{\vs}=\abs{\sv}$ for all $\vs\in\vU$. We often abbreviate this as~$\abs{s}$. An order function on~$\vU$ is \emph{submodular} if \[\abs{\vr \vee \vs}+\abs{\vr \wedge \vs}\le \abs{\vr}+\abs{\vs}\] for all $\vr,\vs\in\vU$.
If a universe comes with a submodular order function, we call it a \emph{submodular universe} of separations.
 
Given a separation system $\vS$, an \emph{orientation} of $S$ is a set $O$ that contains exactly one of $\vs$ and $\sv$ for every $\vs\in \vS$. A \emph{partial orientation} of $S$ is an orientation of any subset of~$S$.

Given a submodular order function on a universe~$\vU$ of separations, a separation system $\vS\sub\vU$, and a number~$k>0$, we denote as $\vS_k$ the separation system consisting of all the separations in~$\vS$ that have order less than~$k$. If $O$ is an orientation of some~$S_k$ and $k'<k$, we say that the orientation $O':=O\cap \vS_{k'}$ of~$\vS_{k'}$ is the \emph{restriction of $O$ to order $k'$}.

In this paper, we consider only the two most basic instances of separation systems: partitions and, more generally, `separations' of sets. Given a set $V\!$, an \emph{unoriented separation} of~$V\!$ is a set $\{A,B\}$ of subsets $A,B$ of $V\!$ such that $A\cup B=V\!$. The ordered pairs $(A,B)$ and $(B,A)$ are the two \emph{orientations} of this separation; the sets $A$ and~$B$ are its {\em sides\/}. Conversely, any ordered pair~$(A,B)$ of subsets of~$V\!$ with $A\cup B = V\!$ is an \emph{oriented separation} of~$V$; then $\{A,B\}$ is its \emph{corresponding unoriented separation}.

The oriented separations of a set come with a natural partial ordering: let $(A,B)\le (C,D)$ if both $A\subseteq C$ and $D\subseteq C$. 
With this partial ordering, and the involution~$^\ast$ mapping every $(A,B)$ to $(B,A)$, they form a separation system. 
The set of \emph{all} the separations of a set $V\!$ form a universe of separations, in which $(A,B)$ and $(C,D)$ have supremum $(A \cup B, C \cap D)$ and infimum $(A \cap B, C \cup D)$. 

The partitions%
   \footnote{Recall that `partitions' in this paper are always partitions into two sets.}
   of a set~$V\!$ form a subset of its separations: those whose sides are disjoint. As the supremum and infimum of two partitions of~$V\!$ are also partitions of~$V\!$, these form their own universe inside the universe of all the separations of~$V\!$.

Generally a \emph{tangle} is an orientation $\tau$ of a separation system $\vS$ which satisfies some additional `consistency conditions'.
The choice of the condition depends on the context, not least on the type of the separation system $\vS$ at hand.
For example, the typical consistency condition for a system of some separations of a set $V\!$ is to ask that for any three oriented separations $(A_1, B_1), (A_2, B_2), (A_3, B_3) \in \tau$ the union $A_1 \cup A_2 \cup A_3$ is not the entirety of $V\!$. This is the condition that we will use in Section~\ref{Tangles}.

Finally, we use the usual graph theoretic notation from \cite{DiestelBook16noEE}. In particular, given a graph $G=(V,E)$, sets $A,B\subseteq V\!$ of vertices, and a vertex $v\in V\!$, we write $N(v)$ for the set of neighbours of~$v$ in~$G$ and $N(A)$ the set of neighbours outside~$A$ of vertices in~$A$. We write $E(A,B)$ the set of edges from $A$ to~$B$, and as $E(A)$ the set $E(A,V\sm A)$ of edges leaving~$A$.

\section{Duality of set separations}\label{SetPartitions}

\subsection{An example of duality for set partitions}\label{example}

Consider a bipartite graph~$G$ with partition classes $X$ and~$Y\!$. For
 $$\vx := \{\,y\in Y\mid xy\in E(G)\,\}\quad\hbox{and}\quad\vy := \{\,x\in X\mid xy\in E(G)\,\}$$
we then have
$$x\in\vy\ \Leftrightarrow\ y\in\vx\eqno (*)$$
for all $x\in X$ and $y\in Y\!$. With $\xv := Y\sm\vx$ and $\yv := X\sm \vy$, the unoriented pairs $\{\vx,\xv\}$ and $\{\vy,\yv\}$ then form partitions of~$Y\!$ and $X$, respectively.

Given~$E(G)$, every $x$ determines the set~$\{\vx,\xv\}$, and every $y$ determines the set~$\{\vy,\yv\}$. If these correspondences are 1--1, i.e., if there are no $x\ne x'$ with $\{\vx,\xv\} = \{\vxdash,\xvdash\}$,\footnote{Note that this can happen in two ways: that $\vx=\vxdash$ or that $\vx=\xvdash$.} and similarly for the~$y$, we may choose to ignore the formal difference between $x$ and~$\{\vx,\xv\}$, and between $y$ and~$\{\vy,\yv\}$. Then $X$ becomes a set of partitions of~$Y\!$, and $Y\!$ a set of partitions of~$X$.

Let us consider $\vx$ as shorthand for the orientation~$(\xv,\vx)$ of the partition $x=\{\vx,\xv\}$ of~$Y\!$ towards its side~$\vx$, and similarly for $\xv$, $\vy$ and~$\yv$. Then
$$\vX := \{\,\vx\mid x\in X\,\}\cup \{\,\xv\mid x\in X\,\}$$
$$\vY := \{\,\vy\mid y\in Y\,\}\cup \{\,\yv\mid y\in Y\,\}$$
 are the sets of all orientations of the elements of~$X$ or~$Y\!$, respectively, and $(\vX,\le)$ and $(\vY,\le)$ form separation systems in the sense of~\cite{ASS}. In view of~$(*)$, we think of these separation systems as \emph{dual} to each other.

Let us next see how to dualise any given system of set partitions: how to define another system of set partitions so that the two are instances of~$(\vX,\le)$ and~$(\vY,\le)$ for a suitable bipartite graph as above.

\subsection{Dualising a given system of set partitions}\label{DualityNaive}

Given any separation system $(\vS,\le)$ consisting of partitions of a set~$V\!$, let us define a separation system $(\vV,\le)$ of partitions of~$S$ that is dual to~$(\vS,\le)$ in the sense of Section~\ref{example}.

We start by picking for every $s\in S$ a default orientation, which we denote as~$\vs$ (rather than~$\sv$). If we think of $\vs = (A,B)$ as the side~$B\sub V\!$ to which it points (so that $\sv$ is equated with~$A$ by the same token), then for every $v\in V\!$ the sets
$$\vv = \{\,s\in S\mid v\in\vs\,\}\quad\hbox{and}\quad\vvback = \{\,s\in S\mid v\in\sv\,\}$$
form a partition of~$S$. Let us assume that these partitions $\{\vv,\vvback\}$ differ for distinct $v\in V\!$, just as the sets $\{\vs,\sv\}$ differ for distinct~$s$ by definition of~$s$.%
   \footnote{Recall from~\cite{ASS} that separation systems are formally defined in such a way that their elements~$\vs$ are given first, and $s$ is then formally defined as~$\{\vs,\sv\}$. Hence if the~$s$ are distinct, as they are here by assumption, then this means that these 2-sets are distinct.}
  Then they determine their $v$ uniquely, and we may think of each $v$ as shorthand for~$\{\vv,\vvback\}$. This makes $V\!$ into a set of partitions $v = \{\vv,\vvback\}$ of~$S$ and
$$\vV := \{\,\vv\mid v\in V\}\>\cup\> \{\,\vvback\mid v\in V\}$$
into the set of all orientations of elements of~$V\!$, and we have
$$v\in\vs\Leftrightarrow s\in\vv\qquad \hbox{as well as}\qquad v\in\sv\Leftrightarrow s\in\vvback\eqno(**)$$
for all the elements $\vs,\sv$ of $\vS$ and $\vv,\vvback$ of~$\vV\!$.

Hence $(\vV,\le)$ and~$(\vS,\le)$ form an instance of a pair of dual separation systems as in Section~\ref{example}, based on a bipartite graph with vertex classes $V\!$ and~$S$ and edges~$vs$ whenever $v\in\vs$ (equivalently, $s\in\vv$).

\subsection{Duality of set separations}
Bipartitions are, and remain, the most important type of separation system whose duals are relevant in applications; our online-shop example from the introduction is a typical example. In order to develop our duality theory, however, it will be easier to work  in the more general context of set separations. These can even have a natural interpretation too: think of~$S$ as a set of questions answered by a set $V\!$ of people, whose answers can be `yes', `no' or `don't know'.

Let us thus adapt the definitions from Section~\ref{DualityNaive} to systems of general set separations rather than just set partitions. Given a system $(\vS, \le)$ of separations of a set~$V\!$, let us define a dual system $(\vV, \le)$ of separations of~$S$ as follows.
We start by picking for every $s\in S$ a default orientation, which we denote as~$\vs$. Unlike in the case of partitions, there is no longer a 1--1 correspondence between the separations $\vs = (A,B)$ and the sides~$B$ to which they point, as the map $\vs\mapsto B$ may fail to be injective. However, if we write $v\in\vs$ for $v\in B$ and $v\in\sv$ for~$v\in A$ informally, then for every $v\in V\!$ the sets
$$C_v = \{\,s\in S\mid v\in\sv\,\}\quad\hbox{and}\quad D_v = \{\,s\in S\mid v\in\vs\,\}$$
form a separation of~$S$. Let us assume that these separations $\{C_v,D_v\}$ differ for distinct $v\in V\!$, just as the sets $\{\vs,\sv\}$ differ for distinct~$s$ by definition of~$s$.
Then for $\vv := (C_v,D_v)$ and $\vvback := (D_v,C_v)$ the set $\{\vv,\vvback\}$ determines $v$ uniquely, and we may think of $v$ as shorthand for~$\{\vv,\vvback\}$. This
makes
$$\vV := \{\,\vv\mid v\in V\}\>\cup\> \{\,\vvback\mid v\in V\}$$
into a separation system whose elements are oriented separations of~$S$. Equating these informally, as above, with the subsets of~$S$ to which they point, we have
$$v\in\vs\Leftrightarrow s\in\vv\qquad \hbox{as well as}\qquad v\in\sv\Leftrightarrow s\in\vvback\eqno(***)$$
for all the elements $\vs,\sv$ of $\vS$ and $\vv,\vvback$ of~$\vV\!$. We may thus call $(\vS,\le)$ and $(\vV,\le)$ {\em dual\/} systems of set separations.

\subsection{The algebraic view}\label{DualityAlgebraic}

In Section~\ref{Hom} we shall point out that the duality defined in Section~\ref{DualityNaive} is an instance of the algebraic duality familiar from algebra or topology. Very roughly, we may consider as a boundary operator the map that sends an oriented separation of a set~$V\!$ to the formal sum of the of elements of~$V\!$ to which it points minus the others. Then the associated coboundary operator performs the dual operation of mapping the elements of~$V\!$ to (co-)chains of oriented separations in a way that the two maps commute, just as in our online shop example.

As we shall note in Section~\ref{sec:innerproduct}, our boundary operator gives rise to an inner product on the separations of~$V\!$ modulo its kernel. The notion of orthogonality  associated with this inner product coincides with geometric orthogonality if our set~$V\!$ is the unit disc and our separations are the partitions of~$V\!$ given by straight lines through the origin. But it is defined for all finite set separations, and thus adds an interesting geometric perspective to their study beyond this particular example.

\section{Tangles from dual separation systems}\label{Tangles}
In our running example of an online shop, we have found that each element $p \in P$ gives rise naturally to a partition of $V\!$, and vice versa. However, these are not the only set of partitions, or separations, of $V\!$ and $P$ that we shall have to consider. 

In order to apply the main tools and theorems of tangle theory to our given sets of partitions of $V\!$ and~$P$, we need to embed them in some richer systems of set partitions or separations. For example, we might need that the separation system we work with is submodular in the sense of~\cite{AbstractTangles}, or even arises as a set of all separations of $V\!$ or~$P$ that have order less than some constant~$k$ with respect to some submodular order function.

It will turn out that the dual nature of our given $V\!$ and $P$ will allow us to define quite a natural order function on all the separations of $V\!$ and~$P$, respectively, which we will prove is submodular. This order function is qualitatively similar to, but not identical to (see \cite{TangleBook}), the order function described in the introduction.

We can then study the tangle structure of the set of separations of $V\!$ or $P$ of order less than $k$ for any fixed $k \in \mathbb{N}$. We will show that the correspondence mentioned in the introduction between arbitrary partitions of $V\!$ and $P$ will extend to a correspondence between the low-order tangles of the separations of $V\!$ and $P$, allowing us to relate the tangle structure of $V\!$ to that of~$P$.

\subsection{Tangles on the sides of a bipartite graph}\label{sec:vtx_tangles_sets}
Let us make the discussion from the start of this section more precise. Suppose that we have a dual pair of separation systems $(\vX,\le)$ and $(\vY,\le)$, whose duality is witnessed by a bipartite graph $G$ with partition classes $X$ and $Y\!$ as in Section~\ref{example}. 

Let us denote by $S(X)$ the set of all separations of the set $X$, that is, the set of all sets $\{A,B\}$ with $A,B\subseteq X$ such that $A\cup B=X$. 
Similarly, we denote by $S(Y)$ the set of all separations of the set $Y\!$, and we denote by $\vS(X)$ and $\vS(Y)$ the set of oriented separations from $S(X)$ or $S(Y)$, respectively.

Then the structure of the bipartite graph $G$, which encodes the duality between $(\vX,\le)$ and $(\vY,\le)$, will allow us to relate the separations in $\vS(X)$ to the separations in $\vS(Y)$. Indeed, each separation in $\vS(X)$ induces a separation in~$\vS(Y)$ and vice versa, in the following manner.
Given a separation $(A,B)$ of $X$ there will be some vertices in $Y\!$ which are joined in $G$ to more vertices in $A$ than in $B$, while other vertices in $Y\!$ are joined to more vertices in $B$ than in $A$. This gives us a natural way to partition the vertices in $Y\!$. So, given $(A,B) \in \vS(X)$ we define the separation $(A,B)^\triangleright:=(A_B^\triangleright,B_A^\triangleright)\in \vS(Y)$ by letting \[A_B^\triangleright:=\{y\in Y \::\: |N(y)\cap A|\ge|N(y)\cap B|\}\] and \[B_A^\triangleright:=\{y\in Y \::\: |N(y)\cap A|\le|N(y)\cap B|\}.\] We call $(A,B)^\triangleright$ the \emph{shift} of $(A,B)$.

Similarly,\footnote{Informally, we think of the vertex classes $X,Y\!$ of $G$ as being its `left' and `right' class, respectively. Formally, however, $\{X,Y\}$ is an unordered pair, so the operators $\cdot^\triangleright$ and $\cdot^\triangleleft$ are formally the same: they map their argument, an oriented separation of one of the sets $X,Y\!$, to an oriented separation of the other set. It is important that we never treat $X$ and $Y\!$ differently in this paper: they are disjoint, but indistinguishable.} a separation $(C,D)$ of $Y\!$ gives rise to a separation of $X$, which we call its \emph{shift},  $(C,D)^\triangleleft:=(C_D^\triangleleft,D_C^\triangleleft)$ via \[C_D^\triangleleft:=\{x\in X \::\: |N(x)\cap C|\ge|N(x)\cap D|\}\] and \[D_C^\triangleleft:=\{x\in X \::\: |N(x)\cap C|\le |N(x)\cap D|\}.\]

We note that both these shifting operations commute with the natural involutions on $\vS(X)$ and $\vS(Y)$. However, we also note that this operation is not necessarily idempotent: there may exist separations $(A,B) \in \vS(X)$ such that ${((A,B)^\triangleright)^\triangleleft\neq (A,B)}$. 

The map $(\cdot)^\triangleright$ between $\vS(X)$ and $\vS(Y)$ induces an inverse `pull-back' map $\pullright(\cdot)$ between the power sets $2^{\vS(Y)}$ and $2^{\vS(X)}$, sending every $\tau\subseteq \vS(Y)$ to \[
\pullright \tau := \{ (A,B) \in \vS(X) \colon (A,B)^\triangleright\in \tau \}\subseteq \vS(X).
\]
Similarly, the map $(\cdot)^\triangleleft\colon\vS(Y)\!\to\vS(X)$ induces a map $\pull(\cdot)\colon2^{\vS(X)}\!\to 2^{\vS(Y)}$ sending every $\tau\subseteq \vS(X)$ to \[
\pull \tau := \{ (C,D) \in \vS(Y) \colon (C,D)^\triangleleft\in \tau \}\subseteq \vS(Y).
\]

The question then arises, under which conditions on a tangle $\tau$ will the subset $\pull \tau$ or $\pullright \tau$ also be a tangle? In order for there to be any interesting tangle structure we will have to restrict to some subset of $\vS(X)$ or $\vS(Y)$, and the most natural way to do so will be to choose some order function and consider the set $\vS_k(X)$ or $\vS_k(Y)$ of separations of order less than $k$. However, then in order for the pullback to have any hope of being a tangle, it must orient every separation in $\vS_{k'}(X)$ or $\vS_{k'}(Y)$ for some $k'$. Hence, already for this question to make sense, we will need to choose an appropriate order function which behaves nicely with respect to the shifting operation.

In fact, we will define order functions on $\vS(X)$ and $\vS(Y)$ so that shifting a separation never increases its order. This will guarantee that if $\tau$ orients all the separations of order less than $k$ in $\vS(X)$, then $\pull \tau$ orients all the separations of order less than $k$ in $\vS(Y)$. Indeed, if $(C,D) \in \vS(Y)$ has order less than $k$, then $(A,B):=(C,D)^{\triangleleft} \in \vS(X)$ has order less than $k$ and so precisely one of $(A,B)$ or $(B,A)$ is in $\tau$ by assumption. Since $(B,A)= (D,C)^{\triangleleft}$ it follows that precisely one of $(C,D)$ or $(D,C)$ is in $\pull \tau$.

Furthermore, these order functions are defined in a particularly natural way, determined only by the structure of $G$. Broadly, the order functions measure in some way how evenly a separation of $X$ or $Y\!$ splits the neighbourhood of each vertex from the appropriate class. For example, in our online shop example, the order of a separation $(A,B)$ of $V\!$ will be determined by how evenly this separation splits the set of items bought in each purchase. The more balanced the split, the larger the contribution of this vertex to the order of the separation. In this way, separations for which most vertices in the opposite partition class have a clear `preference' of one side or the other will have low order.

Explicitly, let us define the order function $\abs{\cdot}_X\colon \vS(X) \to \frac{1}{2}\N$ where
\[\abs{A,B}_X:=\sum_{y\in Y}\left(\min\{\abs{N(y)\cap A},\abs{N(y)\cap B}\} - \abs{N(y) \cap A \cap B} / 2 \right).\]
Here the first term will be larger when $N(y)$ is more evenly split by $(A,B)$. 

The extra term of $-\abs{N(y) \cap A \cap B} / 2$ is to adjust for double-counting: We can think of the term $\min\{\abs{N(y)\cap A},\abs{N(y)\cap B}\}$ as counting the number of edges from $y$ to the smaller of $N(y)\cap A$ and $N(y)\cap B$. However then, if $x$ is contained in both $A$ and $B$, we would count all the edges incident with $x$ in $\sum_{y\in Y}\min\{\abs{N(y)\cap A},\abs{N(y)\cap B}\}$, and so moving $x$ out of $A$ or $B$ would not increase the order even if it made the separation more balanced. With the extra term $-\abs{N(y) \cap A \cap B} / 2$ however, a neighbour $x$ in $A\cap B$ is treated as lying half in $A$ and half in $B$, in the sense that it is counted with a factor of $\frac{1}{2}$ in~$N(y)$.

Similarly, we define $\abs{\cdot}_Y\colon \vS(Y) \to \frac{1}{2}\N$ where
\[\abs{C,D}_Y:=\sum_{x\in X}\left(\min\{\abs{N(x)\cap C},\abs{N(x)\cap D}\} - \abs{N(x) \cap C \cap D} / 2 \right).\]

Note that these functions are symmetric and non-negative, as required of an order function for separation systems. 
Moreover, the function $\abs{\cdot}_X$ attains its maximum value on the separation $(X,X)$, and since orientations of all of $\vS(X)$ are not enlightening, we will in the following assume implicitly that any $\vS_k(X)$ we consider does not contain the separation $(X,X)$.

Less obviously, these order functions are submodular, and so $\vS(X)$ and $\vS(Y)$ equipped with these functions are submodular universes.
Submodularity is a fundamental property for order functions at the heart of tangle theory, and so we include the proof even though it is straightforward. However, the reader is invited to skip the proofs of the next two lemmas at first reading, to remain with the flow of the narrative.

\begin{LEM}\label{lem:order_submodular}
The order function $\abs{\cdot}_X$ is submodular.
\end{LEM}
\begin{proof}
 We show that $\abs{\cdot}_X$ is a sum of submodular functions. For this consider, for $y\in Y\!$, the order function on $\vS(X)$ given by \[\abs{A,B}_y:=\min\{\abs{N(y)\cap A},\abs{N(y)\cap B}\} - \abs{N(y) \cap A \cap B} / 2\] and note that $\abs{A,B}_X=\sum_{y\in Y}\abs{A,B}_y$, thus it is enough to show that $\abs{\cdot}_y$ is submodular for every $y\in Y\!$. Fix some $y$ in $Y\!$. For $Z \subseteq X$ we denote $N_Z := \abs{N(y) \cap Z}$.
 
Let $(A_1,B_1)$ and $(A_2,B_2)$  be separations in $\vS(X)$, and suppose without loss of generality that $N_{A_i} \leq N_{B_i}$. Let $A_i' := A_i \sm B_i$, $B_i' := B_i \sm A_i$ and $Z_i := A_i \cap B_i$. Note that $\abs{A_i,B_i}_y=N_{A_i'}+N_{Z_i}/2$.
 
We observe that \[\abs{A_1 \cap A_2, B_1 \cup B_2}_y=
    N_{A_1' \cap A_2'} + \frac{1}{2}( N_{Z_1 \cap Z_2} + N_{Z_1 \cap A_2'} + N_{A_1' \cap Z_2} ),
 \]
 and \[\abs{A_1 \cup A_2, B_1 \cap B_2}_y=
  \min\{ N_{A_1' \cup A_2'}, N_{B_1' \cap B_2'} \} + \frac{1}{2}( N_{Z_1 \cap Z_2} + N_{Z_1 \cap B_2'} + N_{B_1' \cap Z_2} ).
 \]
 Summing these two, we get
 \begin{align*}
     &\abs{A_1 \cap A_2, B_1 \cup B_2}_y+\abs{A_1 \cup A_2, B_1 \cap B_2}_y\\
     ={}& N_{A_1' \cap A_2'} + \frac{1}{2}( N_{Z_1 \cap Z_2} + N_{Z_1 \cap A_2'} + N_{A_1' \cap Z_2} ) \\
     +{}& \min\{ N_{A_1' \cup A_2'}, N_{B_1' \cap B_2'} \} + \frac{1}{2}( N_{Z_1 \cap Z_2} + N_{Z_1 \cap B_2'} + N_{B_1' \cap Z_2} ) \\
     \leq{}& N_{A_1' \cap A_2'} + N_{A_1' \cup A_2'} \\ 
     +{}& \frac{1}{2}( N_{Z_1 \cap Z_2} + N_{Z_1 \cap A_2'} + N_{Z_1 \cap B_2'} 
     + N_{Z_1 \cap Z_2} +  N_{A_1' \cap Z_2} + N_{B_1' \cap Z_2} ) \\
     ={}& N_{A_1'} + N_{A_2'} + \frac{1}{2}( N_{Z_1} + N_{Z_2} )\\
     ={}&\abs{A_1,B_1}_y+\abs{A_2,B_2}_y.
 \end{align*}
 
Similarly,  \[\abs{A_1 \cap B_2, B_1 \cup A_2}_y=
  \min\{ N_{A_1' \cap B_2'}, N_{B_1' \cup A_2'} \} + \frac{1}{2}( N_{Z_1 \cap Z_2} + N_{Z_1 \cap B_2'} + N_{A_1' \cap Z_2} ).
 \]
 and \[\abs{A_1 \cup B_2, B_1 \cap A_2}_y=
  \min\{ N_{A_1' \cup B_2'}, N_{B_1' \cap A_2'} \} + \frac{1}{2}( N_{Z_1 \cap Z_2} + N_{Z_1 \cap A_2'} + N_{B_1' \cap Z_2} ).
 \]
 Summing these two, we get
 \begin{align*}&\abs{A_1 \cap B_2, B_1 \cup A_2}_y+\abs{A_1 \cup B_2, B_1 \cap A_2}_y\\
    ={}& \min\{ N_{A_1' \cap B_2'}, N_{B_1' \cup A_2'} \} + \frac{1}{2}( N_{Z_1 \cap Z_2} + N_{Z_1 \cap B_2'} + N_{A_1' \cap Z_2} ) \\
     +{}& \min\{ N_{A_1' \cup B_2'}, N_{B_1' \cap A_2'} \} + \frac{1}{2}( N_{Z_1 \cap Z_2} + N_{Z_1 \cap A_2'} + N_{B_1' \cap Z_2} ) \\
     \leq{}& N_{A_1' \cap B_2'} + N_{A_2' \cap B_1'} \\ 
     +{}& \frac{1}{2}( N_{Z_1 \cap Z_2} + N_{Z_1 \cap A_2'} + N_{Z_1 \cap B_2'} 
     + N_{Z_1 \cap Z_2} +  N_{A_1' \cap Z_2} + N_{B_1' \cap Z_2} ) \\
     \le{}& N_{A_1'} + N_{A_2'} + \frac{1}{2}( N_{Z_1} + N_{Z_2} )\\
     ={}&\abs{A_1,B_1}_y+\abs{A_2,B_2}_y.
 \end{align*}
 Thus $\abs{\cdot}_y$ is submodular and so is $\abs{\cdot}_X=\sum_{y\in Y}\abs{\cdot}_y$
\end{proof}

Next, we show that the shifting operation does not increase the order of a separation.
For this we first show the following lemma, giving an alternative representation of the order function: 
\begin{LEM}\label{lem:set_shift_order_fn}
  For all $(A,B) \in \vS(X)$ we have
  \[
    \abs{A,B}_X = \abs{E(A_B^\triangleright, B)} + \abs{E(B_A^\triangleright, A)}
    - \abs{E(A_B^\triangleright \cap B_A^\triangleright, X)} / 2
    - \abs{E(Y, A \cap B)} / 2.
  \]
\end{LEM}
\begin{proof}
  This can be calculated by rearranging sums:
  \begin{align*}
    &\abs{A,B}_X\\
    &={} \sum_{y\in Y}\left(\min\{\abs{N(y)\cap A},\abs{N(y)\cap B}\} - \frac{\abs{N(y) \cap A \cap B}}{2}\right) \displaybreak[0]\\
    &={} \sum_{\substack{y\in Y \\ \abs{N(y)\cap A} \geq \abs{N(y)\cap B}}}
    \abs{N(y)\cap B}+ \sum_{\substack{y\in Y \\ \abs{N(y)\cap B} \geq \abs{N(y)\cap A}}} \abs{N(y)\cap A} \\
    &\qquad - \Big(\sum_{\substack{y\in Y \\ \abs{N(y)\cap A} = \abs{N(y)\cap B}}}
    \frac{\abs{N(y)}}{2} + \sum_{y\in Y} \frac{\abs{N(y) \cap A \cap B}}{2} \Big)\\
    &={} \abs{E(A_B^\triangleright, B)} + \abs{E(B_A^\triangleright, A)} -
    \abs{E(A_B^\triangleright \cap B_A^\triangleright, X)} / 2
    - \abs{E(Y, A \cap B)} / 2.
    \qedhere
  \end{align*}
\end{proof}
With this we can now prove that shifting a separation indeed cannot increase 
the order of a separation:
\begin{LEM}\label{lem:set_fixed_shift_order}
Let $(A,B)$ be a separation of $X$, then $\abs{A,B}_X \geq 
\abs{A^\triangleright_B,B_A^{\triangleright}}_Y$. Similarly if $(C,D)$ is a 
separation of $Y\!$, then $\abs{C,D}_Y \geq 
\abs{C^\triangleleft_D,D^\triangleleft_C}_X$.
\end{LEM}
\begin{proof} This is true by the following calculation:
\begin{align*}
    &\abs{A_B^\triangleright,B_A^\triangleright}_Y\\={}&\sum_{x\in 
X}\left(\min\{\abs{N(x)\cap A_B^\triangleright,N(x)\cap B_A^\triangleright}\}- 
\abs{N(x) \cap A_B^\triangleright \cap B_A^\triangleright} / 2\right)\\
    \le{}&\sum_{a\in A}\abs{N(a)\cap B_A^\triangleright}+\sum_{b\in B} 
\abs{N(b)\cap A_B^\triangleright}-\sum_{x\in A\cap B}\abs{N(x)}/2-\sum_{y\in 
A_B^\triangleright\cap B_A^\triangleright}\abs{N(y)}/2\\
    ={}&\sum_{b\in B_A^\triangleright}\abs{N(b)\cap A}+\sum_{a\in 
A_B^\triangleright} \abs{N(a)\cap B}-\sum_{y\in A_B^\triangleright\cap 
B_A^\triangleright}\abs{N(y)}/2-\sum_{x\in A\cap B}\abs{N(x)}/2\\
    ={}&E(A_B^\triangleright, B) + E(B_A^\triangleright, A) - 
E(A_B^\triangleright \cap B_A^\triangleright, X) / 2
    - E(Y, A \cap B) / 2\\
    ={}&\abs{A,B}_X.\qedhere
    \end{align*}
\end{proof}

Finally, in order to define the tangles of $\vS(X)$ and $\vS(Y)$ we need to define the notion of \emph{consistency} that we require our orientations to satisfy. There are a few natural choices that one could make here, however in most contexts it turns out that these definitions are in some sense weakly equivalent, in that tangles under any one definition tend to induce tangles of slightly lower order under the other definitions.

With that in mind, let us define a \emph{tangle of $\vS_k(X)$ (in $G$)} as an orientation $\tau$ of $\vS_k(X)$ which satisfies the following property: \[
\tag{$\dagger$} \textit{There are no } (A_1,B_1),(A_2,B_2),(A_3,B_3)\in \tau \textit{ with } A_1\cup A_2\cup A_3=X. \label{prop:Xtangle}
\]
We define tangles of $\vS_k(Y)$ in $G$ in a similar manner. This is perhaps the simplest definition to take, and is a direct analogue of the corresponding notion of `consistency' used to define tangles in matroids. We will discuss later in more detail the extent to which our results hold for tangles defined in terms of other notions of `consistency.'

We are now ready to state the main results of this section. We will show that, with the aid of this order function, we can relate the tangles of $\vS(X)$ to those of $\vS(Y)$.

\begin{THM}\label{thm:shifttangle_set}
Let $\tau$ be a tangle of $\vS_{4k}(X)$, then $\tau' := \pull\tau\cap \vS_k(Y)$ is a tangle of $\vS_k(Y)$.
\end{THM}

By symmetry, we then obtain a similar conclusion as in \cref{thm:shifttangle_set} when we shift a $4k$-tangle of $Y\!$.

A natural question then to ask at this point, is, even if the shifting operations themselves are not idempotent, whether the operation they induce on tangles is in some way `idempotent': That is, if we shift a tangle twice, do we end up with the original tangle? It turns out that, again up to a constant factor, this is indeed the case.

\begin{THM}\label{thm:double_shift_set}
    Let $\tau$ be a tangle of $\vS_{16k}(X)$, let $\tau' = \pull\tau \cap \vS_{4k}(Y)$, and let $\tau'' = \pullright\tau'\cap \vS_k(X)$. Then $\tau'' \subseteq \tau$.
\end{THM}

It is possible to prove Theorems \ref{thm:shifttangle_set} and \ref{thm:double_shift_set} directly. However, there is perhaps a more interesting way to prove them indirectly, which potentially gives slightly more illumination to the connection between these two types of tangles, albeit at the cost of a slight increase in the factors of $k$. The idea is to view the tangles of the two partition classes as two different facets of tangles on the edge set of the bipartite graph.

 We give these proofs in the next section. Direct proofs of Theorems \ref{thm:shifttangle_set} and \ref{thm:double_shift_set} are given in the extended version of this paper~\cite{ASSdualityArxiv}.

\subsection{Tangles of the edges}\label{subsec:tangles_edges}
We will show that the tangles on the sides of a bipartite graph are closely related to a third type of tangle - one defined on the separations of the edges. Let us briefly introduce some new notation for this purpose.

We denote the set of all separations of $E$, the edge set of our bipartite graph, as $\vS(E)$, and the set of the corresponding unoriented separations as $S(E)$.
The following order function on the separations in $\vS(E)$ is a natural variation on our previous order function for separations in $\vS(X)$:\[
    \abs{C,D}_E := \sum_{v\in V} \left(\min(\abs{E(v) \cap C}, \abs{E(v) \cap D})-\abs{E(v)\cap C\cap D}/2\right),
\]
where $E(v)$ denotes the set of incident edges of $v$.
We will again assume that any $\vS_k(E)$ we consider does not contain $(E,E)$.

We say that an orientation $\tau$ of a subset $\vS_k(E)$ of $\vS(E)$ is a \emph{tangle} of $\vS_k(E)$, if $\tau$ is an orientation of $\vS_k(E)$ with the following property:
\[
\tag{$\dagger_E$} \textit{There are no } (C_1,D_1),(C_2,D_2),(C_3,D_3)\in \tau \textit{ with } C_1\cup C_2\cup C_3=E. \label{prop:Etangle}
\]

\noindent
Given a separation in $\vS(X)$, there is a reasonably natural candidate for a separation in $\vS(E)$ which it `induces': A separation $(A,B)$ of $X$ naturally defines a separation $(A,B)^E := (E(A), E(B))$ of $E$, where $E(A)$ denotes the set of all edges of $G$ which have an end vertex in $A$. Note that $((A,B)^E)^\ast=(B,A)^E$.
 
 The other way around is less obvious, but it will be necessary to associate to each separation in $\vS(E)$ a separation in $\vS(X)$ and $\vS(Y)$. We will do so similarly to how we associated to each separation in $\vS(Y)$ a separation in $\vS(X)$. There we obtained, given a separation $(A,B)\in \vS(Y)$, a separation in $\vS(X)$ by asking for every vertex in $X$ whether that vertex has more neighbours in $A$ or in $B$. In a similar manner we will now ask, given a separation $(C,D)$ in $\vS(E)$, for each vertex in $X$ whether more of the adjacent edges lie in $C$ or in $D$.  Formally, given a separation $(C,D)$ of $E$, we obtain a separation $(C,D)^\blacktriangleleft := (C^\blacktriangleleft_D,D^\blacktriangleleft_C)$ of $X$ by defining \[
    C^\blacktriangleleft_D = \{x\in X  \::\: \abs{E(x) \cap C} \ge \abs{E(x) \cap D}\}
\]
and 
\[
D^\blacktriangleleft_C = \{x\in X  \::\: \abs{E(x) \cap C} \le \abs{E(x) \cap D}\}.
\]
In an analogous manner we can define a separation $(C,D)^\blacktriangleright$ of $Y\!$, however by the symmetry of the situation we will only ever need to talk about the map~$(\cdot)^\blacktriangleleft$.

Now this shifting operation on the separations preserves the partial order of separations in the following sense:
\begin{LEM}\label{lem:shift_hom}
If $(C,D)\le (C',D')$, then $(C,D)^\blacktriangleleft\le (C',D')^\blacktriangleleft$
\end{LEM}
\begin{proof}
If $(C,D)\le (C',D')$, then $C\subseteq C'$ and $D\supseteq D'$. Thus, for $x\in X$, we have that $\abs{E(x)\cap C}\le \abs{E(x)\cap C'}$ and $\abs{E(x)\cap D}\ge \abs{E(x)\cap D'}$.

Now if $x\in C^\blacktriangleleft_D$, then $\abs{E(x) \cap C} \ge \abs{E(x) \cap D}$ and thus \[\abs{E(x)\cap C'}\ge \abs{E(x)\cap C}\ge \abs{E(x)\cap D}\ge \abs{E(x)\cap D'},\] hence  $x\in C'^\blacktriangleleft_{D'}$. Similarly, if $x\in D'^\blacktriangleleft_{C'}$, then $\abs{E(x)\cap D'}\ge \abs{E(x)\cap C'}$ and thus
\[\abs{E(x)\cap D}\ge \abs{E(x)\cap D'}\ge \abs{E(x)\cap C'}\ge \abs{E(x)\cap C},\] hence  $x\in C^\blacktriangleleft_{D}$. Thus $C^\blacktriangleleft_D\subseteq C'^\blacktriangleleft_{D'}$ and $D'^\blacktriangleleft_{C'}\subseteq D^\blacktriangleleft_C$, i.e. $(C,D)^\blacktriangleleft\le (C',D')^\blacktriangleleft$. 
\end{proof}

Unlike the shifting operations considered in the previous section, there is less of a symmetry here: The separation $(A,B)^E$ fully determines the separation $(A,B)$, whereas the separation $(C,D)^{\blacktriangleleft}$ in some way `compresses' the information in the separation $(C,D)$ into a rough estimate. Generally there are multiple different separations $(C,D)$ in $\vS(E)$ for which the $(C,D)^\blacktriangleleft$ coincide, and so the operation $(\cdot)^\blacktriangleleft$ is not injective.

As with $(\cdot)^\triangleright$, this function induces a pullback map:
given a subset $\tau$ of $\vS(X)$, we define \[\tau_E := \{(C,D)   \::\: (C,D)^\blacktriangleleft\in \tau\}.\] 
Note that, as $(E(A),E(B))^\blacktriangleleft=(A,B)$, the set of all the separations $(A,B)^E$ is a subset of $\tau_E$.

For shifting in the other direction we take a slightly different notion. Given a subset $\tau$ of $\vS(E)$, let us define \[\tau_X:=\{(C,D)^\blacktriangleleft \::\: (C,D)\in \tau\},\]
and let $\tau_Y$ be defined analogously.
Note that this is a genuinely different way to move between tangles of $\vS(E)$ and $\vS(X)$; rather than `pulling back' the tangle from $\vS(X)$ to $\vS(E)$ via the shift~$(\cdot)^E$, giving rise to a set of separations 
\[
\pullE \tau := \{(A,B) \in \vS(X)   \::\: (A,B)^E \in \tau\},
\]
we're `pushing forward' via the shift $(\cdot)^{\blacktriangleleft}$.

We note that in this particular case, since, assuming the graph is connected, it is clear that $\left((A,B)^E\right)^\blacktriangleleft = (A,B)$, we have that $\tau_X \supseteq \pullE \tau$ and so, since $\pullE \tau$ is automatically a partial orientation of $\vS(X)$, if the restriction of $\tau_X$ to some lower order is a tangle, then the restriction of $\pullE \tau$ to the same order will also satisfy~\eqref{prop:Xtangle}. In particular, working with this definition results in slightly stronger results than working with $\pullE \tau$, however the main purpose of this change is that it will result in slightly simpler proofs, see for example \cref{cor:double_shift}.

We will show that given a tangle $\tau$ of $\vS_{4k}(X)$, the set $\tau_E \cap \vS_k(E)$ is in fact a tangle of $\vS_k(E)$ and dually, given a tangle $\tau $ is of $\vS_{2k}(E)$, the set $\tau_X \cap \vS_k(X)$ is a tangle of $\vS_k(X)$. We will then be able to use this to obtain proofs of \cref{thm:shifttangle_set} and \cref{thm:double_shift_set} from the symmetry between $X$ and $Y\!$.

We note first that, as is the case for $\abs{\cdot}_X$ and $\abs{\cdot}_Y$, the order function $\abs{\cdot}_E$ is submodular. 

However, since we will not use this fact, and its proof is almost identical to that of Lemma \ref{lem:order_submodular}, the proof is included only in the extended version of the paper.

\begin{LEM}\label{lem:edge_submod}
The order function  $\abs{\cdot}_E$ is submodular.
\end{LEM}

However, unlike for the correspondence between $\abs{\cdot}_X$ and $\abs{\cdot}_Y$, we will no longer be able to show that the order of the shift of a separation is non-increasing, instead we will only be able to show that, when shifting from a separation of the vertices to the corresponding separation of the edges, we can bound how much the order increases. More precisely, simple calculations show that:
\begin{PROP}\label{lem:order_shift} Given a separation $(A,B)$ of $X$, we have
   $\abs{A,B}_X \le\abs{(A,B)^E}_E$ and  $\abs{(A,B)^E}_E\le 2\abs{A,B}_X$.
\end{PROP}
\begin{proof} For the first statement we note that:
\begin{align*}
\abs{A,B}_X={}&\sum_{y\in Y}\left(\min\{\abs{N(y)\cap A},\abs{N(y)\cap B}\} - \abs{N(y) \cap A \cap B} / 2\right)\\
={}&\sum_{y\in Y}\left(\min\{\abs{E(y)\cap E(A)},\abs{E(y)\cap E(B)}\} - \abs{E(y) \cap E(A) \cap E(B)} / 2\right)\\
\le{}& \sum_{v\in V}\left(\min\{\abs{E(v)\cap E(A)},\abs{E(v)\cap E(B)}\} - \abs{E(v) \cap E(A) \cap E(B)} / 2\right) \\
={}&\abs{E(A),E(B)}_E
\end{align*}
For the second statement we observe that, for $x\in X$ we have, since $x\in A$ or $x\in B$, that \[
\min\{\abs{E(x)\cap E(A)},\abs{E(x)\cap E(B)}\}=\abs{E(x)\cap E(A)\cap E(B)}
\] and thus \[
\sum_{x\in X}\left(\min\{\abs{E(x)\cap E(A)},\abs{E(x)\cap E(B)}\} - \abs{E(x) \cap E(A) \cap E(B)} / 2\right)=\abs{E(A)\cap E(B)}/2.
\]
As clearly $\abs{A,B}_X\ge \abs{E(A)\cap E(B)}/2$ it follows that
\begin{align*}
&\abs{E(A),E(B)}_E\\={}&\sum_{v\in V}\left(\min\{\abs{E(v)\cap E(A)},\abs{E(v)\cap E(B)}\} - \abs{E(v) \cap E(A) \cap E(B)} / 2\right)\\
={}&\sum_{x\in X}\left(\min\{\abs{E(x)\cap E(A)},\abs{E(x)\cap E(B)}\} - \abs{E(x) \cap E(A) \cap E(B)} / 2\right)\\
&+\sum_{y\in Y}\left(\min\{\abs{E(y)\cap E(A)},\abs{E(y)\cap E(B)}\} - \abs{E(y) \cap E(A) \cap E(B)} / 2\right)\\
={}&\abs{E(A)\cap E(B)}/2+\sum_{y\in Y}\left(\min\{\abs{N(y)\cap A},\abs{N(y)\cap B}\} - \abs{N(y) \cap A \cap B} / 2\right)\\
\le{}& 2\abs{A,B}_X \qedhere
\end{align*}
\end{proof}
For $(\cdot)^\blacktriangleleft$ on the other hand, we will be able to show that this is a non-increasing operation:
\begin{LEM} \label{lem:set_edge_decrease}
Let $(C,D)$ be a separation of $E$, then $\abs{C,D}_E \geq \abs{C^\blacktriangleleft_D,D_C^{\blacktriangleleft}}_X$.
\end{LEM}

For the proof of \cref{lem:set_edge_decrease} we will need to carefully analyse how we can `locally' change a separation in $\vS(E)$ without changing the shift.
Recall that, given a separation $(A,B)$ in $\vS(X)$, there are other separations apart from $(A,B)^E$ in $\vS(E)$ which still shift to $(A,B)$. So, in order to prove \cref{lem:set_edge_decrease} we will analyse what these different separations of $E$ inducing the same separation $(A,B)$ of $X$ look like.
For this, we will show which 'local', i.e.\ single-edge, changes we can make to a given separation $(C,D)$ to bring it closer to one of the type $(A,B)^E$, without increasing its order.

So, let us start analysing these `local' changes.
Firstly, in the next lemma we show that we can move a single edge from $C$ to $D$ without increasing the order of $(C,D)$ or changing its shift $(C,D)^\blacktriangleleft$, if at the end vertex in $X$ of that edge there are fewer incident edges in $C$ than in $D$.

\begin{LEM}\label{lem:move_over}
    Let $(C,D)$ be a separation of $E$ and let $e\in E$ be incident with $C^\blacktriangleleft_D\sm D^\blacktriangleleft_C$. Then $\abs{C\cup\{e\},D\sm\{e\}}_E\le \abs{C,D}_E$ and ${(C\cup \{e\},D\sm\{e\})^\blacktriangleleft=(C,D)^\blacktriangleleft}$.
\end{LEM}
\begin{proof}
We may suppose that $e\in D$, as otherwise there is nothing to show. Let $e=vw$. We observe that, since $v\in C_D^\blacktriangleleft\sm D_C^\blacktriangleleft$, we have
\begin{align*}&\min(\abs{E(v) \cap C}, \abs{E(v) \cap D})-\abs{E(v)\cap C\cap D}/2\\=&\abs{E(v)\cap D}-\abs{E(v)\cap C\cap D}/2\\\ge&\abs{E(v)\cap (D\sm\{e\})}-\abs{E(v)\cap (C\cup\{e\})\cap (D\sm\{e\})}/2+1, 
\end{align*}
and
\begin{align*}
&\min(\abs{E(w) \cap C}, \abs{E(w) \cap D})-\abs{E(w)\cap C\cap D}/2\\\ge&\min(\abs{E(w)\cap (D\sm\{e\})},\abs{E(w)\cap (C\cup\{e\})})-\abs{E(w)\cap (C\cup\{e\})\cap (D\sm\{e\})}/2-1. \end{align*}
Thus $\abs{C\cup\{e\},D\sm\{e\}}_E\le \abs{C,D}_E$. Moreover, $v\in C^\blacktriangleleft_D\sm D_C^\blacktriangleleft$ and therefore also $v\in (C\cup\{e\})_{D\sm\{e\}}^\blacktriangleleft \sm  (D\sm\{e\})_{C\cup\{e\}}^\blacktriangleleft$ and thus $(C,D)^\blacktriangleleft=(C\cup\{e\},D\sm\{e\})^\blacktriangleleft$.
\end{proof}

Using this, we can show that from within the set of those separations of $E$ with the same shift $(A,B)$, we can always find some $(C,D)$ of minimal order which is `close' to $(A,B)^E$, in the sense that every edge incident with $A\sm B$ is contained in $C\sm D$ and every edge incident with $B\sm A$ is contained in $D\sm C$:

\begin{LEM}\label{lem:minimal_induced}
Let $(C,D)\in \vS(E)$. Then there exists a separation $(C',D')$ of $E$ with $\abs{C',D'}_E \leq \abs{C,D}_E$ and $(C',D')^{\blacktriangleleft}=(C,D)^{\blacktriangleleft}$ such that every edge $e$ incident with $C_D^\blacktriangleleft\sm D_C^\blacktriangleleft$ lies in $C'\sm D'$ and every edge incident with $D_C^\blacktriangleleft\sm C_D^\blacktriangleleft$ lies in $D'\sm C'$.
\end{LEM}
\begin{proof}

Suppose $(C',D')$ is chosen such that ${(C',D')^{\blacktriangleleft}=(C,D)^{\blacktriangleleft}}$, such that the order of $(C',D')$ is at most the order of $(C,D)$, i.e.\ $\abs{C',D'}_E \leq \abs{C,D}_E$, and such that there are as few edges as possible incident with $C_D^\blacktriangleleft\sm D_C^\blacktriangleleft$ which do not lie in $C'\sm D'$ and as few edges as possible incident with $D_C^\blacktriangleleft\sm C_D^\blacktriangleleft$ which do not lie in $D'\sm C'$.

Suppose that there exists some such edge $e$ incident with $C_D^\blacktriangleleft\sm D_C^\blacktriangleleft$ which does not lie in $C'\sm D'$ or some such edge $e$ incident with $D_C^\blacktriangleleft\sm C_D^\blacktriangleleft$ which does not lie in $D'\sm C'$. Let us assume we are in the former case, as the argument in the latter case is identical.

Since $(C',D')^\blacktriangleleft=(C,D)^\blacktriangleleft$, by \cref{lem:move_over} we could then consider the separation $(C'\cup\{e\},D'\sm\{e\})$ which satisfies ${(C'\cup\{e\},D'\sm\{e\})^\blacktriangleleft=(C',D')^\blacktriangleleft}$ and ${\abs{C'\cup\{e\},D'\sm\{e\}}\le \abs{C',D'}\le \abs{C,D}}$, contradicting the choice of $(C',D')$.
\end{proof}

This observation enables us to perform the necessary calculations to prove \cref{lem:set_edge_decrease}.
\begin{proof}[Proof of \cref{lem:set_edge_decrease}]
By \cref{lem:minimal_induced} we may suppose that every edge incident with $C_D^\blacktriangleleft\sm D_C^\blacktriangleleft$ lies in $C\sm D$ and every edge incident with $D_C^\blacktriangleleft\sm C_D^\blacktriangleleft$ lies in $D\sm C$.

In this case, we can calculate $\abs{C,D}_E$ as follows.
\begin{align*}
&\abs{C,D}_E \\=& \sum_{v\in V} \left(\min(\abs{E(v) \cap C}, \abs{E(v) \cap D})-\abs{E(v)\cap C\cap D}/2\right)\\
=& \sum_{v\in Y} \left(\min(\abs{E(v) \cap C}, \abs{E(v) \cap D})-\abs{E(v)\cap C\cap D}/2\right)+\sum_{v\in C_D^\blacktriangleleft\cap D_C^\blacktriangleleft}\frac{1}{2}\abs{E(v)}\\
\ge& \sum_{v\in Y} \left(\min(\abs{N(v) \cap C_D^\blacktriangleleft}, \abs{N(v) \cap D_C^\blacktriangleleft})-\frac{1}{2}\abs{N(v)\cap C_D^\blacktriangleleft\cap D_C^\blacktriangleleft}\right)\\
&-\frac{1}{2}\abs{E(C_D^\blacktriangleleft)\cap E(D_C^\blacktriangleleft)}+\sum_{v\in C_D^\blacktriangleleft\cap D_C^\blacktriangleleft}\frac{1}{2}\abs{N(v)}\\
\ge& \sum_{v\in Y} \left(\min(\abs{N(v) \cap C_D^\blacktriangleleft}, \abs{N(v) \cap D_C^\blacktriangleleft})-\abs{N(v)\cap C_D^\blacktriangleleft\cap D_C^\blacktriangleleft}/2\right)\\=&\abs{C^\blacktriangleleft_D,D_C^{\blacktriangleleft}}_X. \qedhere
\end{align*}
\end{proof}

Analysing local changes will also play a crucial role in showing that, given a tangle $\tau$ in $\vS(E)$, the restriction of $\tau_X$ to a lower order is actually an orientation. For this we will need to make sure that separations obtained from one another by local changes cannot be oriented differently in $\tau_E$. However, whereas \cref{lem:move_over} allows us to move certain edges from $C\sm D$ to $D\sm C$ without changing the shift or increasing the order, for showing that the restriction of $\tau_X$ to a lower order actually is an orientation we will need to analyse a different type of local change.

More precisely, the next lemma will allow us to move certain edges from $D\sm C$ (or, symmetrically, $C\sm D$), to $C\cap D$ without increasing the order. Such an operation might change the shift of a separation, but, by \cref{lem:shift_hom}, it does so only in a controlled way: moving an edge from $D\sm C$ to $D\cap C$ will only result in a shift that is larger, in the sense of the partial order on the separation system, than the shift of the original $(C,D)$. Moreover, such a local change does not change the way a separation is oriented by a tangle.

\begin{LEM}\label{lem:move_to_middle}
    Let $(C,D)$ be a separation of $E$ and let $e\in E$ be incident with $C_D^\blacktriangleleft$. Then $\abs{C\cup\{e\},D}_E\le \abs{C,D}_E$ and ${(C\cup\{e\},D)^\blacktriangleleft\ge(C,D)^\blacktriangleleft}$.
\end{LEM}
\begin{proof}
If $e\in C$, then it is nothing to show, so suppose that $e\in D\sm C$. Let $e=vw$. We observe that, since $v\in C_D^\blacktriangleleft$, we have
\begin{align*}&\min(\abs{E(v) \cap C}, \abs{E(v) \cap D})-\abs{E(v)\cap C\cap D}/2\\=&\abs{E(v)\cap D}-\abs{E(v)\cap C\cap D}/2\\=&\abs{E(v)\cap D}-\abs{E(v)\cap (C\cup\{e\})\cap D}/2+\frac{1}{2} \\
\intertext{and}
&\min(\abs{E(w) \cap C}, \abs{E(w) \cap D})-\abs{E(w)\cap C\cap D}/2\\\ge&\min(\abs{E(w)\cap (C\cup \{e\})},\abs{E(w)\cap D})-\abs{E(w)\cap (C\cup \{e\})\cap D}/2-\frac{1}{2}. \end{align*}
Thus $\abs{C\cup\{e\},D}_E\le \abs{C,D}_E$. We have $(C,D)^\blacktriangleleft\le(C\cup \{e\},D)^\blacktriangleleft$ by \cref{lem:shift_hom}.
\end{proof}

We now have all the ingredients at hand needed to show that the shift of a tangle, restricted to an appropriate order, is still a tangle. Let us start by considering the shift $\tau_X$ of a tangle $\tau$ of $\vS(E)$.

\begin{THM}\label{prop:edges_to_vtx}
   If $\tau$ is a tangle of $\vS_{2k}(E)$, then $\tau_X\cap \vS_k(X)$ is a tangle of~$\vS_k(X)$.
\end{THM}
\begin{proof}
We first note that the set $\tau_X\cap \vS_k(X)$ contains at least one of $(A,B)$ and $(B,A)$ for every separation $(A,B)\in \vS_k(X)$. Indeed, by Proposition \ref{lem:order_shift} $\abs{(A,B)^E}_E \leq 2 \abs{A,B}_X$, and so since $\tau$ is a tangle of $\vS_{2k}(E)$ either $(A,B)^E \in \tau$ or $(B,A)^E \in \tau$. 

Let us now show that for no separation $\{A,B\}$ we have both $(A,B)$ and $(B,A)$ in $\tau_X\cap \vS_k(X)$.
Suppose otherwise, then $\tau$ contains separations $(C_1,D_1)$ and $(C_2,D_2)$ such that $(C_1,D_1)^\blacktriangleleft=(A,B)$ and $(C_2,D_2)^\blacktriangleleft=(B,A)$.

 Note that, by \cref{lem:order_shift}, we have that ${\abs{E(A),E(B)}_E\le 2\abs{A,B}_X<2k}$, hence $(A,B)^E\in \tau$ or $(B,A)^E\in \tau$.
 As $(E(A),E(B))^\blacktriangleleft=(A,B)$, we may suppose without loss of generality that either $(C_1,D_1)=(A,B)^E$ or ${(C_2,D_2)=(B,A)^E}$. We suppose the former one, the latter case is similar.
 
 Now pick a separation $(C,D)\in \tau$ such that $(C,D)^\blacktriangleleft\ge (C_2,D_2)^\blacktriangleleft=(B,A)$ and the set $(D_1\cap D)\sm (C_1\cup C)$ is as small as possible. Then, since $\tau$ satisfies \eqref{prop:Etangle}, we have $C_1\cup C\neq E$. Hence there exists some edge $e\in (D_1\cap D)\sm (C_1\cup C)$. 
 
  Let $x$ be the end vertex of $e$ in $X$. Note that $e\in E(B)\sm E(A)$ since $E(A)=C_1$ and $e\notin C_1$. Thus $x\in B\sm A$. Moreover, as $B={C_2}_{D_2}^\blacktriangleleft\subseteq C_D^\blacktriangleleft$, we have that $x\in C_D^\blacktriangleleft$.
  Thus $e$ is incident with $C_D^\blacktriangleleft$.
  
  Consequently, we can apply \cref{lem:move_to_middle} to get that $\abs{C\cup \{e\},D}_E\le \abs{C,D}_E$. Thus $\tau$ orients $(C\cup \{e\},D)$ and therefore $(C\cup \{e\},D)\in\tau$, as $(D,C\cup \{e\})\in \tau$ would contradict \eqref{prop:Etangle} because of $(C,D)\in \tau$ and $D\cup C=E$.
  
  But this implies that $(C\cup \{e\},D)\in \tau$ is a better choice for $(C,D)$, since $(C\cup \{e\},D)^\blacktriangleleft\ge (C,D)^\blacktriangleleft$ by \cref{lem:shift_hom} and 
  \[
  (D_1\cap D)\sm (C_1\cup C)\supsetneq (D_1\cap D)\sm (C_1\cup (C\cup \{e\}))
  \]
  as $e\in (D_1\cap D)\sm (C_1\cup C)$.

Thus $\tau_X\cap \vS_k(X)$ is indeed an orientation.
That $\tau_X\cap \vS_k(X)$ satisfies the tangle property~\eqref{prop:Xtangle} now follows like this: If $(A_1,B_1),(A_2,B_2),(A_3,B_3)$ would be a triple in $\tau_X\cap \vS_k(X)$ contradicting the tangle property \eqref{prop:Xtangle}, then $\tau$ would need to orient $(A_1,B_1)^E$,$(A_2,B_2)^E$ and $(A_3,B_3)^E$ by \cref{lem:order_shift}. By the above observation $\tau$ orients them as  $(A_1,B_1)^E$,$(A_2,B_2)^E$ and $(A_3,B_3)^E$, since $\tau_X\cap \vS_k(X)$ does not contain any $(((A_i,B_i)^E)^\ast)^\blacktriangleleft=(B_i,A_i)$. However, the three separations $(A_1,B_1)^E$,$(A_2,B_2)^E$ and $(A_3,B_3)^E$ in $\tau$ then contradict the tangle property \eqref{prop:Etangle}, as every edge in $E$ is incident with at least one of the sets $A_1,A_2,A_3$. 
\end{proof}

A similar conclusion holds for the shift $\tau_E$ of a tangle $\tau$ of $\vS_{4k}(X)$.

\begin{THM}\label{prop:vtx_to_edges}
Given a tangle $\tau$ of $\vS_{4k}(X)$, then $\tau_E\cap \vS_k(E)$ is a tangle of~$\vS_k(E)$.
\end{THM}
\begin{proof}
By \cref{lem:set_edge_decrease}, given some separation $(C,D)\in \vS_k(E)$ we have that $\abs{(C,D)^\blacktriangleleft}_X\le \abs{(C,D)}_E$, thus $\tau$ contains exactly one of the two separations $(C,D)^\blacktriangleleft$ and ${((C,D)^\blacktriangleleft)^\ast=(D,C)^\blacktriangleleft}$, and consequently $\tau_E\cap \vS_k(E)$ contains exactly one of $(C,D)$ and $(D,C)$, i.e.\ $\tau_E\cap \vS_k(E)$ is an orientation of $\vS_k(E)$.

So, it remains to show that $\tau_E\cap \vS_k(E)$ satisfies the tangle property~\eqref{prop:Etangle}. Let us suppose for a contradiction that there is some set \[\{(C_1,D_1),(C_2,D_2), (C_3,D_3)\} \subseteq \tau_E\cap \vS_k(E)\] such that
$C_1 \cup C_2 \cup C_3 = E$.

Let $(A_i,B_i) = (C_i,D_i)^\blacktriangleleft$ for each $i=1,2,3$. Then, since $(A_i,B_i) \in \tau$ for each $i$, and $\tau$ is a tangle, it follows that the set $Z = X \setminus (A_1 \cup A_2 \cup A_3)$ is non-empty.

Since $Z \subseteq B_i = {D_i}^\blacktriangleleft_{C_i}$ for each $i$, we have that $|E(z) \cap D_i| \geq |E(z) \cap C_i|$ for all $z \in Z$ and $i = 1,2,3$. However, since $C_1 \cup C_2 \cup C_3 = E$,
\begin{align*}
    \sum_{i=1}^3 \abs{C_i,D_i}_E &= \sum_{i=1}^3 \sum_{v \in V} \left(\min \{ \abs{E(v) \cap C_i}, \abs{E(v) \cap D_i}\} - \frac{1}{2}\abs{E(v) \cap C_i \cap D_i}\right) \\
    &\geq \sum_{i=1}^3 \sum_{z\in Z}\left( \min \{ \abs{E(z) \cap C_i}, \abs{E(z) \cap D_i}\} - \frac{1}{2}\abs{E(z) \cap C_i \cap D_i} \right)\\
     &=  \sum_{z\in Z} \sum_{i=1}^3 \left( \abs{E(z) \cap C_i} - \frac{1}{2}\abs{E(z) \cap C_i \cap D_i}\right)\\
     &\geq  \sum_{z\in Z} d(z)/2 = \abs{E(Z,Y)}/2.
\end{align*}
As $\abs{C_i,D_i}_E<k$ for every $i=1,2,3$, this gives us $\abs{E(Z,X)} < 6k $ and thus
\begin{align*}
\abs{Z,X}_X = \abs{E(Z,Y)} / 2 < 3k.
\end{align*}
Hence, $\tau$ needs to orient $(Z,X)$. As $(X,Z)\in \tau$ would contradict \eqref{prop:Xtangle}, it follows that $(Z,X) \in \tau$.

Finally, since $\abs{A_3,B_3}_X \leq \abs{C_3,D_3}_E < k$ by \cref{lem:set_edge_decrease}, we can conclude by submodularity, that
\[ \begin{split}
\abs{A_3 \cup Z,B_3 \cap X}_X \leq \abs{A_3, B_3}_X + \abs{Z,X}_X
< 4k.
\end{split} \]
Hence, it follows that $\tau$ needs to orient $(A_3\cup Z,B_3)$ and as $(A_3,B_3)\in \tau$ it follows from \eqref{prop:Xtangle} that $(A_3\cup Z,B_3)\in \tau$, as $A_3\cup B_3=X$. However, then 
\[
\{ (A_1,B_1),(A_2,B_2),(A_3 \cup Z,B_3)\} \subseteq \tau
\]
and $A_1 \cup A_2 \cup (A_3 \cup Z) = X$, contradicting \eqref{prop:Xtangle}.
\end{proof}
\begin{COR}\label{cor:double_shift}
Let $\tau$ be a tangle of $\vS_{8k}(E)$, then 
\[
\tau'' := \left(\tau_X\cap \vS_{4k}(X)\right)_E\cap \vS_k(E)
\]
is a subset of $\tau$. Similarly, let $\tau'$ be a tangle of $\vS_{8k}(X)$, then \[
\tau''' := \left(\tau'_E\cap \vS_{2k}(E)\right)_{X}\cap \vS_k(X)
\]
is a subset of $\tau'$.
\end{COR}
\begin{proof}
By \cref{prop:edges_to_vtx}, $\tau''$ is a tangle of $\vS_k(E)$.
Now, given any separation $(C,D)\in \vS_k(E)\cap \tau$, we have that $(C,D)^\blacktriangleleft\in \tau_X\cap \vS_{4k}(X)$ and thus $(C,D)$ is in $\tau''$.  As $\tau''$ is an orientation of $\vS_k(E)$, we then have that $\tau''\subseteq \tau$.

For the second part we note that, by \cref{prop:vtx_to_edges}, $\tau'''$  is a tangle of $\vS_k(X)$.
Given $(A,B)\in \vS_k(X)\cap \tau'$ we have, since $((A,B)^E)^\blacktriangleleft=(A,B)$, that $(A,B)^E$ is in $\tau'_E\cap \vS_{2k}(E)$ and thus $(A,B)$ is in $\tau'''$. As $\tau'''$ is an orientation of $\vS_k(X)$, we then have that $\tau'''\subseteq \tau'$.
\end{proof}
Putting these together, we obtain versions of \cref{thm:shifttangle_set} and \cref{thm:double_shift_set}, with slightly worse factors:

\begin{COR}
\label{thm:shifttangle_weaker}
Let $\tau$ be a tangle of $\vS_{8k}(X)$. Then $\tau' := \pull\tau\cap \vS_k(Y)$ is a tangle of $\vS_k(Y)$.
\end{COR}
\begin{proof}
It is easy to see that $\tau'=\left(\tau_E\cap \vS_{2k}(E)\right)_Y\cap \vS_k(Y)$ which is a tangle by \cref{prop:vtx_to_edges} and \cref{prop:edges_to_vtx}.
\end{proof}

\begin{COR}
\label{thm:double_shift_weaker}
    Let $\tau$ be a tangle of $\vS_{64k}(X)$, let $\tau' = \pull\tau\cap \vS_{8k}(Y)$, and let $\tau'' = \pullright\tau'\cap \vS_k(X)$. Then $\tau'' \subseteq \tau$.
\end{COR}
\begin{proof}
Consider $\tau_E\cap \vS_{16k}(E)$. By \cref{thm:shifttangle_weaker}, we have that
\[
\tau'=\left(\tau_E\cap \vS_{16k}(E)\right)_Y\cap \vS_{8k}(Y).
\] 
Moreover, again by \cref{thm:shifttangle_weaker} we have that 
\[
\tau''=\left(\tau'_E\cap \vS_{2k}(E)\right)_X\cap \vS_k(X).
\]
But now, by \cref{cor:double_shift}, we note that 
\[
\left(\left(\tau_E\cap \vS_{16k}(E)\right)_Y\cap \vS_{8k}(Y)\right)_E\cap \vS_{2k}(E)\subseteq \tau_E\cap \vS_{16k}(E)
\] 
and thus, 
\begin{align*}
\tau''&=\left(\left(\left(\tau_E\cap \vS_{16k}(E)\right)_Y\cap \vS_{8k}(Y)\right)_E\cap \vS_{2k}(E)\right)_X\cap \vS_k(X)\\
&\subseteq \left(\tau_E\cap \vS_{16k}(E)\right)_X\cap \vS_k(X).
\end{align*}
Again by \cref{cor:double_shift} we have that 
\[
\left(\tau_E\cap \vS_{16k}(E)\right)_X\cap \vS_k(X)\subseteq \tau,
\] 
which shows the claim.
\end{proof}

\subsection{Variations, Generalisations and open problems}\label{sec:generalisations}
A natural question to consider at this point is how much these results depend on the very specific set up we have here. 

For example, whilst we considered a very specific type of tangle, there are other types of `tangle-like' clusters which one might wish to consider. Perhaps the most general condition one could consider here would be that of a \emph{regular profile}. A \emph{profile} of a set separation system is an orientation which neither contains any two separations pointing away from each other -- i.e., if ${(A_1,B_1)\le (A_2,B_2)}$ for distinct separations $\{A_1,B_1\},\{A_2,B_2\}$ then a profile cannot contain both, $(B_1,A_1)$ and $(A_2,B_2)\in P$ --
nor does it contain any triple of separations of the form
\[
\{ (A_1,B_1),(A_2,B_2), (B_1\cap B_2,A_1\cup A_2) \}.
\]
A \emph{regular} orientation is one which does not contain any cosmall separations, that is a separation $(V,B)$ where $V\!$ is the underlying set and $B \subseteq V\!$. Tangles are regular profiles, but regular profiles model a broader class of clusters. 

Similar statements as in Theorems \ref{thm:shifttangle_set} and \ref{thm:double_shift_set} can be shown to hold via similar arguments for regular profiles. In particular the following Theorems hold:
\begin{THM}\label{thm:shiftprofile}
Let $P$ be a regular profile of $\vS_{3k}(X)$, then $P' := \pull P\cap \vS_k(Y)$ is a regular profile of $\vS_k(Y)$.
\end{THM}

\begin{THM}\label{thm:double_shift_profile}
    Let $P$ be a regular profile of $\vS_{9k}(X)$, $P' = \pull P\cap \vS_k(Y)$ and $P'' = \pullright P'\cap \vS_k(X)$, then $P'' \subseteq P$.
\end{THM}

We can also deduce statements analogue to \cref{prop:edges_to_vtx} and \cref{prop:vtx_to_edges} for profiles, i.e. we can consider regular profiles on subset of the set $\vS(E)$ of separations of the edges of our bipartite graph and show the following:
\begin{PROP}\label{prop:edges_to_vtx_profiles} 
   Let $P$ be a regular profile of $\vS_{2k}(E)$, and let us denote ${P_X:=\{(C,D)^\blacktriangleleft \::\: (C,D)\in P\}}$. Then, the set $P_X\cap \vS_k(X)$ is a regular profile of $\vS_k(X)$.
\end{PROP}

\begin{PROP}\label{prop:vtx_to_edges_profiles}
Let $P$ be a regular profile of $\vS_{3k}(X)$ and let us define ${P_E := \{(C,D) \in \vS_{3k}(E)  \::\: (C,D)^\blacktriangleleft \in P\}}$.
Then $P_E\cap \vS_k(E)$ is a regular profile of $\vS_k(E)$.
\end{PROP}

However, since the arguments are similar to those in \cref{subsec:tangles_edges} we refer interested readers to the extended version of this paper.

Another possible variation of the problem is to consider other ways to relate tangles of the different systems to each other. Given our shifting operation between the two separation systems $\vS(X)$ and $\vS(Y)$ we defined a `pull-back' type operation that maps subsets of $\vS(X)$ to subsets of $\vS(Y)$ and investigated its action on tangles. However, as in the definition of $\tau_X$ there is another way to extend our shifting operations from acting on single separations to acting on subsets via a `push-forward' type action. It is perhaps equally natural to ask how the tangles of $\vS_k(X)$ and $\vS_k(Y)$ behave under these operations. 

Given a tangle $\tau$ of $\vS_k(X)$ one may define the set \[
    \tau^\triangleright := \{(A,B)^\triangleright\::\: (A,B) \in \tau\} \subseteq \vS_k(Y),
\] and similarly, if $\tau$ is a tangle of $\vS_k(Y)$, we may define  
\[
    \tau^\triangleleft := \{(A,B)^\triangleleft\::\: (A,B) \in \tau\} \subseteq \vS_k(X).
\]
Note that $\tau^\triangleright$ and $\tau^\triangleleft$, generally, are no more than subsets of $\vS_k(Y)$ or $\vS_k(X)$, respectively, they need not be an orientation, not even a partial orientation.

However, we can show that this push-forward $\tau^\triangleright$ is, when restricted appropriately, contained in a corresponding pull-back $\pull \tau$ and thus needs to be a partial orientation satisfying \eqref{prop:Xtangle}.

\begin{PROP}
   Let $\tau$ be a tangle of $\vS_{16k}(X)$, then \[
        (\tau \cap \vS_k(X))^\triangleright \subseteq \pull \tau.
   \]
\end{PROP}
\begin{proof}
The only way in which this may fail is that for some $(C,D) \in \pull \tau$ we have $(D,C) \in (\tau \cap \vS_k(X))^\triangleright$.
Let us say this happens because of some ${(A,B) \in  \tau \cap \vS_k(X)}$ with $(A,B)^\triangleright = (D,C)$.

Then also $(A,B) \in \pullright{(\pull\tau\cap \vS_{4k}(Y))}\cap \vS_k(X)$ by \cref{thm:double_shift_set}, and hence $(A,B)^\triangleleft = (D,C) \in  \pull\tau\cap \vS_{4k}(Y)$, contradicting the fact that $\pull\tau \cap \vS_{4k}(Y)$ is a tangle.
\end{proof}

A third variation of this idea is motivated by applications, see for example \cite{TangleClusteringWeakStrong}. There we often wish to work with systems of set partitions, rather than more general set separations. Again here much of the work in previous sections remains true in this setting, with slight tweaks to the definitions and results.

More explicitly, given as before a bipartite graph $G$ on partition classes $X$ and $Y\!$ let $\vcB(X)$ and $\vcB(Y)$ be the universe of all the partitions of $X$ and~$Y\!$, respectively.

Given a partition $(A,B)$ of $X$, we can define, as before, the shift of $(A,B)$ to be the partition $(C,D)$ of $Y\!$ where $C$ is the set of all elements of $Y\!$ with more neighbours in $A$ than in $B$ and $D$ is the set of all elements of $Y\!$ with more neighbours in $B$ than in $A$. However, a small issue arises here as to what to do with those vertices which have an equal number of neighbours in $A$ and $B$. Since we need the shift of a partition to be a partition we need to break the symmetry in some way here and we define our shifting operation not for unoriented, but for oriented partitions, namely we define a partition $(A,B)^\triangleright:=(C,D)$ of $Y\!$ by letting \[C:=\{y\in Y \::\: |N(y)\cap A|\ge|N(y)\cap B|\}\] and \[D:=\{y\in Y \::\: |N(y)\cap A|<|N(y)\cap B|\}.\]
In particular, in general this operation may not commute with the involutions on $\vcB(X)$ and $\vcB(Y)$, i.e., it may be the case that $(A,B)^\triangleright\neq ((B,A)^\triangleright)^\ast$. 

There is again a natural order function for these partitions given by 
\[
\abs{A,B}_X \coloneqq \sum_{y\in Y}\min\{\abs{N(y)\cap A},\abs{N(y)\cap B}\},
\] 
which can again be seen to be submodular, and for a suitable definition of a tangle we can show that analogues of Theorems \ref{thm:shifttangle_set} and \ref{thm:double_shift_set} hold for tangles of $\vcB(X)$ and $\vcB(Y)$. In particular the following is true:

\begin{restatable}{THM}{thmshiftbip}\label{thm:shift_bipart}
Let $\tau$ be a tangle of $\vcB_{4k}(X)$, then $\tau' := \pull\tau\cap \vcB_k(Y)$ is a tangle of $\vcB_k(Y)$.
\end{restatable}

\begin{restatable}{THM}{thmbipdouble}\label{thm:double_shift_bipart}
    Let $\tau$ be a tangle of $\vcB_{16k}(X)$, let $\tau' = \pull\tau \cap \vcB_{4k}(Y)$, and let $\tau'' = \pullright\tau'\cap \vcB_k(X)$, then $\tau'' \subseteq \tau$.
\end{restatable}

Again the arguments closely follow the proofs of Theorems \ref{thm:shifttangle_set} and \ref{thm:double_shift_set}, so we refer an interested reader to the extended version of this paper.

It would be nice if one could find a unified result implying these different variations. Unfortunately, it seems that the nature of the result means that strengthening or weakening the notion of tangle we consider does not make the statement stronger or weaker, but rather incomparable. Indeed, since we wish to show that tangles on $X$ shift to tangles on $Y\!$, if we consider a stronger notion of tangle, then fewer orientations are tangles, and so it is required to show that a stronger property holds for the shifts, but under a stronger assumption on the original orientations. Similarly, if we consider a weaker notion of tangles, then more orientations will be tangles, and so it is required to show that a weaker property holds for the shifts, but we only have weaker assumptions on the original orientations.

A similar problem arises if one wants to relate the statements for set partitions to the statements about set separations: In principle, every tangle of separations of a set induces a tangle of its partitions. And conversely, every tangle of partitions of a set induces a tangle of separations of that set of lower order, except that the `regularity' conditions of these two types of tangles are not compatible: For set separations we just require that we do not contain any cosmall separations, whereas for partitions we want more, namely that the big side of our partition of the vertices in $X$ not only contains more than one vertex, but also its neighbourhood contains at least two vertices from $Y\!$. 
Thus the statements for these two types of tangles are, formally, independent from each other, although most of the proof strategy is very similar.

\section{A homological view of duality}\label{Hom}

Our aim in this section is to point out that the duality of set partitions outlined in Section~\ref{DualityNaive} can be viewed naturally as a case of algebraic duality. The algebraic approach leads to some immediate questions, whose interpretations for set partitions illuminates these from a new angle.

Although motivated by the naive duality of set partitions as in Section~\ref{DualityNaive}, we shall set up our algebraic framework for the slightly larger class of set separations, in line with the rest of this paper. The homology of set separations is tantamount to that of hypergraphs~\cite{HypergraphHomology}, to which we refer for algebraic details.

In this section, $S$~is always a set of {\em oriented\/} separations of some fixed finite set~$V\!$.%
  \footnote{We shall never consider unoriented separations in a homological context such as here, and need to fix default orientations anyhow, as in Section~\ref{DualityNaive}. So we may as well drop the arrows.}
  We assume that $S$ is antisymmetric: that no two elements of~$S$ are inverses of each other as oriented separations. The reason for this restriction is not that we cannot deal with pairs of inverse separations, but that we do not want the inverse of a separation~$s$ to carry an independent name that cannot be traced back to~$s$. If necessary, we can always refer to the inverse of~$s$ as~$s^*$, but in our algebraic context this will not normally be necessary.%
  \footnote{Elements of~$S$ will occur as 1-chains, and thus have a natural inverse in this chain group. If desired, one may think of the chain~$-s$ as the separation inverse to~$s$.}
   Let us call the elements of~$S$ the {\em default orientations\/} of their underlying unoriented separations of~$V\!$.
   \looseness=-1

\subsection{Chains and cochains}\label{sec:chains}

Let $C_0$, $C_1$ and $C_2$ denote the free abelian groups, or $\Z$-modules, with bases $V\!$, $S$, and~$\es$, respectively. We write the elements of~$C_0$, the 0-{\em chains\/} of our separation system, as sums $\sum n_i v_i$ of elements~$v_i$ of~$V\!$ with integer coefficients $n_i\in\Z$. The elements of~$C_1$, its 1-{\em chains\/}, are the sums $\sum n_i s_i$ of separations~$s_i\in S$ with integer coefficients. The set $C_2 = \{0\}$ of `2-chains' consists only of the empty sum.\looseness=-1

As {\em boundary homomorphisms\/} we take the map $\partial_2\colon 0\mapsto 0$ from $C_2$ to~$C_1$ and the homomorphism $\partial_1\colon C_1\to C_0$ that sends every $s=(A,B)\in S$ to the 0-chain $\sum_{v\in B\sm A} v - \sum_{v\in A\sm B} v$ (and extends linearly to all of~$C_1$). An arbitrary~$v\in V\!$ thus has a coefficient in~$\partial_1(A,B)$ of $-1$ if $v\in A\sm B$, of~0 if $v\in A\cap B$, and of~1 if $v\in B\sm A$. In the informal language of Section~\ref{SetPartitions}, the boundary of a separation~$s$ assigns the coefficient~1 to the elements of~$V\!$ that $s$ points to, coefficient~$-1$ to those it points away from, and~0 to those in the middle (Figure~\ref{bdry}).

\begin{figure}[ht]
 \center
   \includegraphics[scale=1]{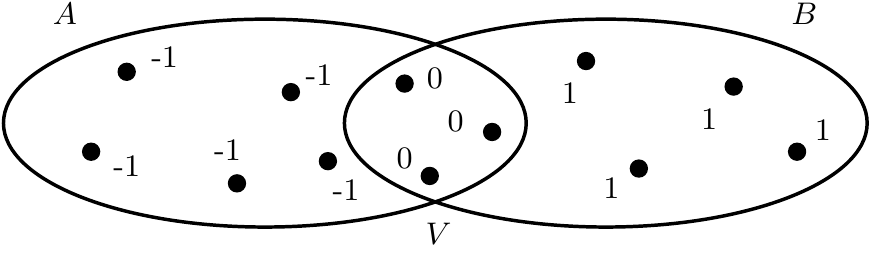}
\caption{Boundary coefficients for a separation~$s=(A,B)$}
\label{bdry}
\end{figure}

If desired, the boundary homomorphism~$\partial_1$ can be described by an $n\times m$ matrix $B$ with entries $1,0,-1$, where $n=|V|$ and $m=|S|$. For a 1-chain $x = \sum_{j=1}^m \alpha_j s_j$ we then have $\partial_1 x = \sum_{i=1}^n \beta_i v_i$ for $Ba=b$, where $a = (\alpha_1,\dots,\alpha_m)$ and $b = (\beta_1,\dots,\beta_n)$.

Graph homology is a special case of this. Indeed, given a graph~$G$ on~$V\!$, define for every oriented edge $ab$ of~$G$ the separation $s_{ab} = (A,B)$ of~$V\!$ given by $A\sm B = \{a\}$ and $B\sm A = \{b\}$, and let $S = \{\,s_{ab} \mid ab\in G\,\}$. Then $\partial_1 (s_{ab}) = b-a$, which is the usual boundary of the edge~$ab$ in simplicial homology. Conversely, any system~$S$ of separations of a set~$V\!$ as considered here can be viewed as a hypergraph on~$V\!$, with oriented edges~$(A\sm B, B\sm A)$ for the separations $(A,B)$ of~$V\!$ in~$S$. So the homology of set separations reduces to hypergraph homology, which is studied in its own right, and from a graph-theoretic perspective, in~\cite{HypergraphHomology}.

For $n=0,1,2$, the elements of $C^n = {\rm Hom}(C_n,\Z)$ are the $n$-cochains of our separation system. We usually define these homomorphisms explictly only on the singleton chains in~$C_n$, i.e., on individual elements of~$V\!$ or of~$S$ when $n=0$ or $n=1$, and extend these maps linearly to all of~$C_n$. We write $C^2 = \{0\}$, where 0~denotes the unique homomorphism $0\mapsto 0$ from $C_2 = \{0\}$ to~$\Z$.

The {\em coboundary homomorphisms\/} for $n=0,1$ are the maps $\delta^n\colon  C^n\to C^{n+1}$ that send an $n$-cochain $\varphi\colon C_n\to\Z$ to the $(n+1)$-cochain $(\varphi\circ\partial_{n+1})\colon C_{n+1}\to\Z$. For example, $\delta^0$~sends $\varphi\in C^0$ to the unique homomorphism $\psi\colon C_1\to\Z$ in~$C^1$ that maps every $s\in S\sub C_1$ to the image of~$\partial_1 s$ in~$\Z$ under~$\varphi$. If $\partial_1$ is described by a matrix~$B$ as earlier, then $\delta^0$ is described by its transpose~$B^\top$. We write~$\varphi_v$ for the element of~$C^0$ that sends $v\in V\!$ to~1 and all the other elements of~$V\!$ to~0, and $\psi_s$ for the element of~$C^1$ that sends $s\in S$ to~1 and all the other elements of~$S$ to~0.

Chains can be turned into cochains simply by interpreting their coefficients as images of the basis element to which they are assigned. Let $\gamma_0$ be the isomorphism $C_0\to C^0$ that maps every $v\in V\!$ to~$\varphi_v$, and $\gamma_1$ the isomorphism $C_1\to C^1$ that maps every $s\in S$ to~$\psi_s$. Thus, $\gamma_0$ maps $\sum_{v\in V} n_v v\in C_0$ to the 0-cochain that sends each $v\in V\!$ to~$n_v$ (and extends linearly to all of~$C_0$), and similarly for~$\gamma_1$.

To avoid clutter, we shall drop the indices 0, 1 or~2 of the maps $\partial$, $\delta$ and~$\gamma$ when they can be understood from the context, as in Figure~\ref{chain diagram}.

\begin{figure}[ht]
 \center
\begin{tikzcd}
0
\arrow[r] 
& C_1
\arrow[r, "\partialone"] \arrow[d, "\gammaone"]
& C_0
\arrow[r] \arrow[d, "\gammanought"]
&0 \\
0
& C^1 \arrow[l]
& C^0 \arrow[l, "\deltanought"]
& 0 \arrow[l]
\end{tikzcd}
\caption{Boundary and coboundary maps}
\label{chain diagram}
\end{figure}

For $x\in C_1$ we shall use the abbreviations of
$$\varphi_{\partial x} := (\gammanought\circ \partialone)(x)\in C^0\quad\text{and}\quad \psi_{\partial x} := (\deltanought\circ\gamma\circ\partial)(x) = \delta(\varphi_{\partial x}) \in C^1.$$
For $s\in S\sub C_1$ we thus have $\varphi_{\partial s} (v) = 1$ if $s$ points towards~$v$, and $\varphi_{\partial s} (v) = -1$ if $s$ points away from~$v$. This is compatible with our earlier definition of $\varphi_v$, since if $\partial x = v\in V\sub C_0$ then $\varphi_{\partial x}$ coincides with~$\varphi_v$ as defined earlier.

\medbreak

In Section~\ref{SetPartitions} we saw that, if $S$ consists of partitions of~$V\!$, the elements of~$V\!$ could in turn be viewed as partitions of~$S$ so that the (oriented) `partition' $v$ of~$S$ points to $s\in S$ if and only if the (oriented) partition $s$ of~$V\!$ points to~$v$.

Our homological setup reflects this informal duality formally in the duality between the boundary and coboundary operators $\partial$ and~$\delta$, and extends it in this form to set separations that are not partitions. This is done in Proposition~\ref{vs} below. Its proof is straightforward from the definitions, but since these are somewhat complex, we provide it for the convenience of readers unfamiliar with homological terms.

\begin{PROP}\label{vs}
For all $v\in V\!$ and $s\in S$, the coefficient of~$v$ in~$\partialone(s)\in C_0$ equals the coefficient of~$s$ in $(\gammaone^{-1}\circ\deltanought\circ\gammanought)(v)\in C_1$.
\end{PROP}

\begin{proof}
  Let us determine the coefficient of~$s$ in $(\gammaone^{-1}\circ\deltanought\circ\gammanought)(v) = \gammaone^{-1}(\deltanought(\gammanought(v)))$. By definition of~$\gamma_1$, it is the image of~$s\in C_1$ in~$\Z$ under $\deltanought(\gammanought(v))\in C^1$. This is, by definition of~$\deltanought$, the image of $\partialone s\in C_0$ under~$\gammanought(v)\in C^0$.

The 0-chain $\partialone s$ assigns coefficients of~$-1$, 0 or~1 to all the elements of~$V\!$, as in Figure~\ref{bdry}. However only the coefficient of~$v$ in the 0-chain~$\partialone s$ matters when $\gammanought(v)\in {\rm Hom}(C_0,\Z)$ is applied to it, since this homomorphism sends all $v'\in V\sm\{v\}$ to~0. Hence the image of $\partialone s$ under $\gammanought(v)$ is just the product in~$\Z$ of~$\gammanought(v)(v)$, which is~$1\in\Z$, with the coefficient of~$v$ in~$\partialone s$, and thus equals this coefficient (as claimed).
  \end{proof}

It is not clear, but maybe worth following up, how much homology theory can inform the theory of tangles beyond the observation of Proposition~\ref{vs}. Every orientation of a set~$S$ of separations of a set~$V\!$, and in particular every tangle~$\tau$ of~$S$, can be viewed as a chain $x = \sum_{s\in S} \lambda_s s$ or a cochain $\psi\colon s\mapsto\lambda_s$, with $\lambda_s = 1$ if $\tau$ orients $s$ in its default direction and~$-1$ otherwise.

Can~$H_1$, which in our case is just the kernel of~$\partial_1$, contain any tangles? A~simple double-counting argument shows that it cannot contain tangles~$\tau$ that have a {\em decider\/}~\cite{Focus,weighted_deciders_AIC}, a~set $Z\sub V\!$ such that $\tau$ orients every $s\in S$ towards the side that contains more of~$X$ than the other side, possibly calculated with weights. But it might be instructive to compare the tangles of~$S$ to the cosets in~$C_1/\Ker\partial$. This quotient is isomorphic to the annihilator of~$\Ker\partial_1$ in~$C^1$, which contains the image of~$\delta^0$ in~$C^1$ and is often equal to it~\cite{HypergraphHomology}.

Conversely, given a 0-chain $\varphi\colon V\to\Z$ let $\tau_\varphi$ be the partial orientation of~$S$ that orients each $s\in S$ by its default orientation if~$\varphi(\partial s) > 0$, by its other orientation if $\varphi(\partial s) < 0$, and not at all if $\varphi(\partial s) = 0$. If $\tau_\varphi$ orients all of~$S$, then $\varphi$ is, in tangle terms~\cite{Focus}, a decider for~$\tau_\varphi$.

But the coboundary $\psi = \varphi\circ\partial$ preserves more of the information inherent in~$\varphi$ than $\tau_\varphi$ does. For example, if $\psi(s)$ is not just positive but large, because the $v\in V\!$ on the side of~$s$ to which it points carry a lot of weight under~$\varphi$, then this says more than just knowing that $\tau_\varphi$ orients~$s$ forward. This additional information might be relevant when we seek to understand tangles with deciders.

If we are interested specifically in finding tangles with deciders, our problem is to determine the coboundaries $x=\sum \lambda_s s$ with $\lambda_s\in \{-1,1\}$ for which there exists $\varphi\in C^0$ such that $\lambda_s \varphi(\partialone s) > 0$ for all~$s$. Equivalently, if we write $B$ for the $(|V|\times |S|)$-matrix with entries $-1,0,1$ that describes~$\partialone$, we are looking for $\lambda\in\{-1,1\}^{|S|}$ and $\mu\in\Z^{|V|}$ or~$\mu\in\R^{|V|}$ such that $[B\lambda]^t\mu > 0$. This might be approachable from an optimisation point of view, perhaps with some constraints on the weights $\mu_v$ of $\varphi = \sum \mu_v v$ such as $\mu_v\ge 0$ or $\mu_v\in\{0,1\}$ or $\sum\mu_v = 1$, or by relaxing the requirement that $\lambda\in\{-1,1\}^{|S|}$. Which of these~$\lambda$ define tangles then still has to be seen, and will depend on the notion of tangle used, but every tangle with a decider will come from such a~$\lambda$.

\subsection{An inner product for set separations}\label{sec:innerproduct}

Given a chain~$x\in C_1$, consider $\psi_{\partial x} = (\deltanought\circ\gammanought\circ\partialone)(x)$. This is a 1-cochain, a homomorphism $C_1\to\Z$. Thus, $\psi_{\partial x}$ maps every chain $y\in C_1$ to some integer $\psi_{\partial x}(y)\in\Z$. Let us denote this integer by
$$\langle x,y \rangle_\partial := \psi_{\partial x}(y)\in\Z.$$
The form $\langle\ ,\ \rangle_\partial$ is easily seen to be bilinear. It is also symmetric:

\begin{LEM}{\rm\cite{HypergraphHomology}}\label{symmetry}
Let $x,y\in C_1$. Write $\alpha_v$ and~$\beta_v$ for the coefficients of all the~$v\in V\!$ in $\partialone x$ and $\partialone y$, respectively, so that $\partialone x = \sum_{v}\alpha_v v$ and $\partialone y = \sum_{v}\beta_{v} v$. Then $\langle x,y\rangle_\partial = \sum_{v\in V} \alpha_v\beta_v = \langle y,x\rangle_\partial$.
\end{LEM}

Figure~\ref{fig:innerproduct} illustrates Lemma~\ref{symmetry} when $x$ and~$y$ are single separations of~$V\!$. In this case $\langle x,y \rangle_\partial$ measures how similar $x$ and~$y$ are: it counts the $v\in V\!$ on which $x$ and $y$ agree (by both pointing towards~$v$ or both pointing away from~$v$) and deducts the number of~$v\in V\!$ on which they disagree, while ignoring those~$v$ that lie in the middle for at least one of the separations $x$ and~$y$.

\begin{figure}[ht]
 \center
   \includegraphics[scale=1]{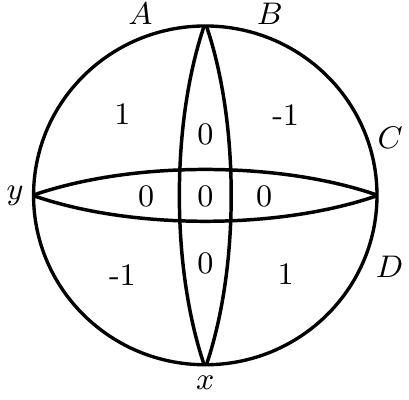}
\caption{The values of~$\alpha_v \beta_v$ for $x = (A,B)$ and $y = (C,D)$}
\label{fig:innerproduct}
\end{figure}

The intuition that when $x,y\in C_1$ are single separations then $\langle x,y \rangle_\partial$ measures their similarity is borne out most clearly when $x$ and~$y$ are partitions of~$V\!$. For $x=y$ we then have $\langle x,y \rangle_\partial = |V|$; when $x$ and~$y$ cross symmetrically as in Figure~\ref{fig:innerproduct} we get $\langle x,y \rangle_\partial = 0$; and when $y$ is the inverse of~$x$ (as a separation)%
  \footnote{By our convention, $S$~does not contain inverse pairs of separations, so this does not actually happen. But the point here is to illustrate the definition of~$\langle\ ,\ \rangle_\partial$ as given explicitly above. And that could be applied to a pair of inverse separations.}
   the definition of~$\langle\ ,\ \rangle_\partial$ gives $\langle x,y \rangle_\partial = \langle x,-x \rangle_\partial = -|V|$.

It is perhaps instructive to compare this with the canonical inner product, or `dot product', $\langle\ ,\ \rangle$ on~$C_1$. For $x,y\in C_1$ with $x = \sum_i \lambda_i s_i$ and $y = \sum_i \mu_i s_i$, where the $s_i$ run over~$S$, we have $\langle x,y \rangle = \gammaone(x)(y) = \sum_i \lambda_i\mu_i$. When $x,y\in S$, then this just indicates whether or not $x$ and~$y$ are identical: $\langle x,y \rangle = \delta_{xy}$. As a similarity measure for two separations $x$ and~$y$, it is rather cruder than~$\langle x,y \rangle_\partial$ which, by Lemma~\ref{symmetry}, is the dot product of $\partialone x$ and $\partialone y$ rather than of $x$ and~$y$.

Our $\langle\ ,\ \rangle_\partial$ is not an inner product on~$C_1$: when $x$ or $y$ is a {\em cycle\/}, an element of $\Ker\partialone$, then clearly $\langle x,y \rangle_\partial = 0$. However, setting $\langle [x],[y]\rangle_\partial := \langle x,y \rangle_\partial$  yields a well-defined symmetric bilinear form on~$C_1/\Ker\partialone$, which is an inner product there:

\begin{LEM}{\rm\cite{HypergraphHomology}}\label{lem:innerproduct}
On $C_1/\Ker\partialone$, the form $\langle\ ,\ \rangle_\partial$ is an inner product.
  \end{LEM}

It would be interesting see how the definition of~$\langle\ ,\ \rangle_\partial$ plays out in some concrete natural separation systems~-- e.g., whether orthogonality has a natural interpretation in that context. Here is one example.

When $S$ is the set of partitions of the unit disc~$D$ by straight lines through the origin, orthogonality for $\langle\ ,\ \rangle_\partial$ agrees with geometric orthogonality. Indeed, if $V\!$ is some suitable discretisation of~$D$ then, as Figure~\ref{fig:innerproduct} indicates, we have $\langle x,y \rangle_\partial = 0$ for two such partitions $x$ and~$y$ if and only if $V\!$ has equally many elements in the two opposite quadrants marked~1 as in the two quadrants marked~$-1$. Clearly, this will be the case if and only if the straight lines through the origin that determine the partitions $x$ and~$y$ are geometrically orthogonal.

\medbreak

If we think of $\langle\ ,\ \rangle_\partial$ as taking real values, its associated norm $\|x\| = \sqrt{\langle x,x\rangle_\partial}$ is also interesting when $x$ is a single separation. Indeed, for $s = (A,B)\in S$ and $\partialone s = \sum_v \alpha_v v$ we have $\langle s,s\rangle_\partial = \sum_v \alpha_v^2$ with $\alpha_v = 0$ for $v\in A\cap B$ and $\alpha_v^2=1$ for all other~$v$. Hence $\|s\| = \sqrt{|V|-|s|}$, where $|s|:= |A\cap B|$ is the order of the separation~$s$.

In other words, for general set separations our algebraic setup encodes their standard submodular order function $s\mapsto |s|$ without this having been written into the setup in any way.

Our function $\langle s,s\rangle_\partial$ is also submodular (as well as supermodular) on~$S$, if $S$ is a sublattice of the lattice (or {\em universe\/})~$U$ of all the separations of~$V\!$. Indeed, if for all $r,s\in S$ their infimum $r\land s$ and their supremum $r\lor s$ in~$U$ is again in~$S$, and thus has a norm, then
$$\|r\land s\|^2 + \|r\lor s\|^2 = 2|V| - |r\land s| - |r\lor s| =  2|V|  - |r| - |s| = \|r\|^2 + \|s\|^2 $$
for all $r,s\in S$.

The norm $\|\ \|$ itself is not submodular on~$S$: one can easily construct separations $r,s$ of~$V\!$ such that $\|r\land s\| + \|r\lor s\| > \|r\| + \|s\|$. However, as we shall see in a moment, we do have for~$\|\ \|$ the most important structural implication of the submodularity of an order function, which is that for all $r,s$ of order less than some integer~$k$ either $r\land s$ or $r\lor s$ also has order~$< k$.

Separation systems with this property, i.e., which contain for any two of their elements also their infimum or supremum in some given universe, are called (structurally) {\em submodular\/}~\cite{ASS}. In our context, the separation systems
$$S_{k,\|\ \|} := \{\,s\in S : \|s\| < k\,\}$$
for fixed $k\in\N$ are submodular:

\begin{PROP}
If $S$ contains both $r\land s\in U$ and $r\lor s\in U$ whenever $r,s\in S$, then for every integer~$k$ the separation system $S_{k,\|\ \|}\sub S$ is submodular.
\end{PROP}

\proof As $\langle\ ,\ \rangle_\partial$ is submodular, we have $\|r\land s\|^2 + \|r\lor s\|^2 \le \|r\|^2 + \|s\|^2$ for all $r,s\in S$. Assume that $\|r\|^2\le \|s\|^2$. Then  $\|r\land s\|^2 \le \|s\|^2$ or $\|r\lor s\|^2 \le \|s\|^2$, and correspondingly $\|r\land s\| \le \|s\|$ or $\|r\lor s\| \le \|s\|$. Hence if $r$ and~$s$ lie in~$S_{k,\|\ \|}$ then so does $r\land s$ or~$r\lor s$.
  \endproof

\bibliographystyle{plain}
\bibliography{collective}

\end{document}